\documentclass[reqno,11pt,letterpaper]{amsart}
\usepackage{amsmath, amsfonts, amsthm, amscd, amssymb}
\usepackage[T1]{fontenc}
\usepackage{lmodern}
\usepackage{thmtools, thm-restate}
\usepackage[backend=biber, style=numeric]{biblatex}
\usepackage{indentfirst, bm, mathtools}
\usepackage{tikz}
\usepackage{fancyhdr}
\usepackage[colorlinks=true]{hyperref}

\newtheorem{theorem}{Theorem}[section]
\newtheorem{proposition}{Proposition}[section]
\newtheorem{lemma}{Lemma}[section]
\newtheorem{corollary}{Corollary}[section]
\newtheorem{example}{Example}[section]

\theoremstyle{definition}
\newtheorem{definition}{Definition}[section]

\theoremstyle{remark}
\newtheorem*{remark}{Remark}
\numberwithin{equation}{section}

\usetikzlibrary{arrows.meta, calc, positioning, shapes, decorations.markings,decorations.pathmorphing, bending}
\definecolor{surfaceblue}{RGB}{91,153,194}
\definecolor{surfaceedge}{RGB}{30,79,113}
\definecolor{glueboundary}{RGB}{119,74,151}

\tikzset{
    surface/.style      = {draw = surfaceedge, line width    = 0.9pt, top color = surfaceblue!8, bottom color = surfaceblue!42, line join = round},
    contour/.style      = {draw = surfaceedge!65, line width = 0.45pt, opacity  = 0.55},
    seam/.style         = {draw = glueboundary, line width   = 1.7pt},
    title/.style        = {font = \large, text               = surfaceedge},
    object label/.style = {text = surfaceedge},
    note/.style         = {text = black!68, align            = center}
}

\DeclareMathOperator{\area}{Area}
\DeclareMathOperator{\vol}{Vol}
\DeclareMathOperator{\rank}{rank}
\DeclareMathOperator{\sing}{Sing}
\DeclareMathOperator{\reg}{Reg}
\DeclareMathOperator{\Trap}{Trap}
\DeclareMathOperator{\Hess}{Hess}
\DeclareMathOperator{\im}{im}

\begin{document}
\title{The Ergodicity of Geodesic Flows on Rank One Manifolds of Nonpositive Curvature}

\author{Fei Liu}
\address{College of Mathematics and Systems Science, Shandong University of Science and Technology, Qingdao, 266590, China}
\email{liufei@math.pku.edu.cn}

\author{Xiaokai Liu}
\address{Department of Mathematics, Southern University of Science and Technology, Shenzhen, 518055, China}
\email{liuxk2024@mail.sustech.edu.cn}
\thanks{Corresponding Author: Xiaokai Liu (liuxk2024@mail.sustech.edu.cn)}

\begin{abstract}
    In this paper, we study the ergodicity of geodesic flows on closed rank-one manifolds of nonpositive sectional curvature. We introduce the infinite-order vanishing set of the Gaussian curvature in dimension two and of the fiberwise second moment of the reduced Jacobi determinant in higher dimensions. We bound the Liouville measure of the singular set in terms of the volume of this subset and, for surfaces, obtain a corresponding Hausdorff dimension bound. In particular, if this subset has zero volume, then the singular set has zero Liouville measure, and the geodesic flow is ergodic.

    We also give a geometric criterion for finitely many exceptional regions whose fundamental groups have virtually Abelian images of rank smaller than the dimension of the manifold. Under strict negativity of sectional curvature outside these regions, we obtain a Hausdorff dimension bound for the singular set and prove that the geodesic flow is ergodic with respect to Liouville measure.

    By extending this method to real-analytic metrics, we characterize the singular set as the zero set of a nontrivial real-analytic determinant constructed from curvature operators, thereby establishing ergodicity in this setting without further assumptions.

    \textbf{Keywords}: geodesic flows; ergodicity; rank-one manifolds; singular set; infinite-order vanishing of curvature.

    \textbf{2020~MSC}: 37D40, 37D25
\end{abstract}
\maketitle
\begin{center}
    \textit{Dedicated to Professor Lan Wen on the occasion of his eightieth birthday.}
\end{center}

\section{Introduction}
The relationship between curvature and the behavior of the geodesic flow is a classical theme in differential geometry and dynamical systems. On closed surfaces of constant negative curvature, Hedlund first established ergodicity with respect to the Liouville measure~\cite{Hedlund1934}. Hopf later proved the corresponding result for surfaces of variable negative curvature~\cite{Hopf1939}. In the 1960s, the work of Anosov and Sinai established uniform hyperbolicity and ergodicity for geodesic flows on closed negatively curved manifolds in arbitrary dimension~\cite{Anosov1967, AnosovSinai1967}.

On manifolds of nonpositive curvature, the presence of zero curvature allows parallel normal Jacobi fields and makes the rank a fundamental geometric invariant. A geodesic has rank one if it does not admit nonzero parallel normal Jacobi fields, and a manifold is rank one if it contains at least one rank-one geodesic. The structural theory developed by Ballmann, Brin, and Eberlein established the geometric framework for studying these manifolds~\cite{BallmannBrinEberlein1985}. The rank rigidity theorem of Burns and Spatzier shows that a closed manifold of nonpositive curvature with irreducible universal cover and rank at least two is locally symmetric~\cite{BurnsSpatzier1987}. Thus, rank-one manifolds provide a natural setting for studying which dynamical properties of negative curvature persist in the presence of zero curvature.

For a closed rank-one manifold of nonpositive curvature, write $\reg$ for the set of rank-one vectors, which is open, dense and invariant under the geodesic flow; and write $\sing$ for the complement, which is closed and also invariant. Eberlein characterized the uniform hyperbolicity by the absence of singular vectors~\cite{Eberlein1973}. For ergodicity with respect to the Liouville measure, the relevant obstruction is the Liouville measure of the singular set. The geodesic flow is ergodic if and only if $\sing$ has zero Liouville measure~\cite{Burns1983, BurnsMatveev2021}. However, it remains unknown in general whether the singular set of a rank-one manifold has measure zero, even for surfaces. Burns-Matveev discuss this problem and explain the gaps in earlier claims of ergodicity~\cite[Section~6.1]{BurnsMatveev2021}; also see~\cite{Wu2015}.

Several geometric hypotheses have been introduced to imply ergodicity. For surfaces, Wu proved that if the negative-curvature locus on the surface has only finitely many connected components, then the singular set has measure zero~\cite{Wu2015}. Wu, Liu and Wang extended the ergodicity criterion to surfaces without focal points~\cite{WuLiuWang2023}. For real-analytic rank-one surfaces, the singular set is known to be a finite union of periodic orbits corresponding to closed flat geodesics, see Knieper~\cite[Section~14]{Hasselblatt2007} and Burns, Climenhaga, Fisher and Thompson~\cite[Section~10.1]{BurnsClimenhagaFisherThompson2018}.

Substantial progress has also been made for other invariant measures. Knieper proved uniqueness of the measure of maximal entropy for closed rank-one manifolds~\cite{Knieper1998}. Burns, Climenhaga, Fisher and Thompson proved uniqueness of equilibrium states for suitable potentials whenever the pressure of the singular set is strictly smaller than the full pressure~\cite{BurnsClimenhagaFisherThompson2018}. Call and Thompson obtained the Kolmogorov property under this pressure gap and the Bernoulli property for the measure of maximal entropy~\cite{CallThompson2022}. Recently, Call, Constantine, Erchenko, Sawyer and Work established local product structure for a class of equilibrium states in rank-one manifolds~\cite{CallConstantineErchenkoSawyerWork2026}. These developments describe strong statistical properties of the geodesic flow, but do not show the ergodicity with respect to Liouville measure in general.

In this paper, we study the singular set from a different perspective. Rather than focusing directly on the geometry of the singular set, which can be difficult to control, especially in higher dimensions, we investigate an approach from a function-theoretic perspective and relate the singular set to the infinite-order vanishing of the curvature quantities on the base manifold. In particular, for surfaces, this set is precisely the set of points at which the Gaussian curvature vanishes to infinite order. We prove that the singular fibers, away from the infinite-order vanishing set, are finite, and consequently, the geodesic flow is ergodic when this set has zero volume. As an application, we construct a smooth genus-two surface with an ergodic, non-Anosov geodesic flow and infinitely many connected components of the negative-curvature locus. We also give a geometric criterion for finitely many exceptional regions whose fundamental groups have virtually Abelian images in the fundamental group of the manifold. If the virtual Abelian ranks are smaller than the dimension of the manifold and the sectional curvature is strictly negative outside these exceptional regions, the singular set has zero Liouville measure. More precisely, we bound its Hausdorff dimension in terms of these ranks. This criterion applies even when the zero curvature locus have positive volume, and therefore complements our function-theoretic criterion.

Finally, for real-analytic metrics in arbitrary dimensions, we construct a real-analytic determinant whose zero set is precisely the singular set. This yields a direct proof of the ergodicity without any additional assumptions.

The paper is organized as follows. Section~\ref{sec2} introduces the notation and states our main results. Section~\ref{sec3} treats smooth surfaces and Section~\ref{sec4} presents the example. Section~\ref{sec5} establishes the curvature criterion in higher dimensions and Section~\ref{sec6} gives the geometric criterion for virtually Abelian exceptional regions. Finally, Section~\ref{sec7} treats real-analytic rank-one manifolds.

\section{Notation and Main Results}\label{sec2}
Let $(M,g)$ be a closed, connected, smooth Riemannian manifold of dimension $n\geq 2$ and nonpositive sectional curvature. We assume $M$ has rank one, as defined below. In particular, when $\dim M=2$, we assume it is a smooth orientable closed surface with genus at least $2$. Whenever we say $M$ is analytic, the Riemannian metric $g$ is understood to be real analytic.

We denote the unit tangent bundle by $T^{1}M$, and $\pi: T^{1}M\to M$ is the standard projection map. Denote the Riemannian curvature tensor by $R$ and the Gaussian curvature function by $K$ on surfaces. For orthogonal unit vectors $\bm{v},\bm{w}$, denote the sectional curvature by $K(\bm{v},\bm{w})=\langle R(\bm{w},\bm{v})\bm{v},\bm{w}\rangle$.

On $T^1M$, let $\mu_{L}$ denote the Liouville measure, normalized to a probability measure. We denote the Riemannian volume on $M$ by $\vol$ and in dimension $2$, by $\area$.

For each $\bm{v}\in T^{1}M$, let $c_{\bm{v}}:\mathbb{R}\to M$ be the unique unit-speed geodesic determined by
\begin{displaymath}
    c_{\bm{v}}(0)=\pi(\bm{v}),\quad c'_{\bm{v}}(0)=\bm{v}.
\end{displaymath}

The geodesic flow is given by
\begin{displaymath}
    \phi_t(\bm{v})=c'_{\bm{v}}(t).
\end{displaymath}

For $\bm{v}\in T^1M$, define its rank, $\rank(\bm{v})$, to be one plus the dimension of the vector space of parallel normal Jacobi fields along $c_{\bm v}$. The rank of the manifold is defined by

\begin{displaymath}
    \rank(M)=\min\left\{\rank(\bm{v})\mid \bm{v}\in T^{1}M\right\}.
\end{displaymath}

Therefore, the rank of a manifold is at least $1$. The manifold is rank one when some geodesic admits no nonzero parallel normal Jacobi field. It is well-known that by Gauss-Bonnet Theorem, a closed surface of genus at least $2$ with nonpositive curvature cannot be flat, thus there is at least one point of strictly negative Gaussian curvature. All tangent vectors on geodesics through this point have rank $1$ (cf.~\cite{Ballmann1985}).

The set of higher-rank vectors (rank at least $2$) is usually called the singular set, while the set of rank $1$ vectors is called the regular set. We denote
\begin{displaymath}
    \sing =\left\{\bm{v}\in T^{1}M \mid \rank(\bm{v})\ge 2\right\},\qquad \reg = T^{1}M\setminus \sing.
\end{displaymath}

The singular set is closed and invariant under the geodesic flow, and the regular set is open and invariant. In particular, on surfaces, the singular set consists precisely of the unit vectors tangent to `flat geodesics':

\begin{displaymath}
    \Lambda =\left\{\bm{v}\in T^{1}M \mid K(c_{\bm{v}}(t))=0,~\forall t\in\mathbb{R}\right\}=\sing.
\end{displaymath}

We recall two classic criteria for hyperbolicity and ergodicity of the geodesic flow on manifolds of nonpositive curvature:

\begin{theorem}\label{thm_2_1}
    The geodesic flow $\phi_t$ is Anosov if and only if $\sing$ is empty (cf.~\cite{Eberlein1973}).
\end{theorem}

\begin{theorem}\label{thm_2_2}
    The geodesic flow $\phi_t$ is ergodic with respect to the Liouville measure $\mu_L$ if and only if $\mu_L(\sing)=0$ (cf.~\cite{Burns1983, BurnsMatveev2021}).
\end{theorem}

Hence, on a rank-one manifold $M$, the existence of even a single geodesic with higher rank tangent vectors precludes uniform hyperbolicity of the geodesic flow. In particular, on surfaces, a flat geodesic will prevent the geodesic flow from being Anosov. From the measure-theoretic viewpoint, however, isolated flat geodesics are negligible. Liouville ergodicity is obstructed precisely by the presence of a positive Liouville-measure singular set.

To study the ergodicity of geodesic flow on a surface of nonpositive curvature and genus at least $2$, rather than studying directly the singular set $\sing \subset T^{1}M$, we focus on the zero-curvature locus on the base surface. More precisely, we investigate to what extent points at which the Gaussian curvature vanishes to infinite order can control the fibers, measure, and Hausdorff dimension of the singular set, and hence the ergodic behavior of the geodesic flow.

\begin{definition}[Infinite Flat Set]\label{def_2_1}
    Let $M$ be a closed surface of nonpositive curvature and genus at least $2$. Define
    \begin{displaymath}
        F_{\infty}\coloneq\left\{x\in M \mid \nabla^{j}K(x)=0, \forall j\geq 0\right\}.
    \end{displaymath}

    Here $\nabla^{j} K$ is the $j$-th covariant derivative of the Gaussian curvature function $K$ with the convention $\nabla^{0}K=K$.
\end{definition}

Our first result gives a sufficient condition for ergodicity in terms of the area of this set.
\begin{restatable}{theorem}{surface}\label{thm_2_3}
    On a closed, connected, smooth surface of nonpositive curvature and genus at least $2$, if $\area(F_{\infty})=0$, then $\phi_t$ is ergodic with respect to the Liouville measure.
\end{restatable}

For a real-analytic metric, the Gaussian curvature is also real analytic. The preceding theorem therefore gives the following corollary.

\begin{restatable}{corollary}{analysurface}\label{cor_2_1}
    On a closed, connected, real-analytic surface of nonpositive curvature and genus at least $2$, the geodesic flow $\phi_t$ is ergodic with respect to the Liouville measure.
\end{restatable}

This corollary is already known. For a rank-one surface with a real-analytic metric, the singular set is a finite union of periodic orbits (cf.~\cites[Section~10.1]{BurnsClimenhagaFisherThompson2018}[Section~14]{Hasselblatt2007}), which has zero Liouville measure, leading to the ergodicity. Alternatively, this result can also be deduced from~\cite[Theorems~1.3 and~1.6]{Wu2015}. Our purpose here is to give a different approach, and for this reason we record the result in this form.

In higher dimensions, the surface argument does not extend simply by replacing the Gaussian curvature $K$ with the curvature tensor $R$. The reason is that a singular vector requires only the existence of one nonzero parallel normal Jacobi field along its geodesic; it does not require the full curvature tensor to vanish along that geodesic. We therefore introduce an analogue of $F_{\infty}$ based on the Jacobi operator, denoted by $F_{\infty}^{J}$ and defined in Section~\ref{sec5}. We achieve an analogous result in higher dimensional manifolds:

\begin{restatable}{theorem}{manifold}\label{thm_2_4}
    On a closed, connected, smooth, rank-one manifold $M$ of nonpositive sectional curvature, let $\mu_L$ be the normalized Liouville measure. Then
    \begin{displaymath}
        \mu_{L}(\sing)\leq \dfrac{\vol(F^{J}_{\infty})}{\vol(M)}.
    \end{displaymath}

    In particular, if $\vol(F^{J}_{\infty})=0$, the geodesic flow is ergodic with respect to the Liouville measure.
\end{restatable}

The preceding sufficient condition does not apply when the flat locus contains an open region, since such a region is contained in $F_{\infty}^{J}$ and contributes positive volume. Our next result gives a geometric criterion in terms of the images of the fundamental groups of exceptional regions. Its proof is given in Section~\ref{sec6}.

\begin{restatable}{theorem}{virtuallyabelian}\label{thm_2_5}
    On a closed, connected, smooth, rank-one manifold $M$ of nonpositive sectional curvature, let $\Omega_1,\ldots,\Omega_N$ be pairwise disjoint compact smooth regions in it, where $N\geq 1$. Assume that:
    \begin{enumerate}
        \item All sectional curvatures at each point of $M\setminus\bigcup_{i=1}^N\Omega_i$ are strictly negative.
        \item For each $i$, the image
            \begin{displaymath}
                H_i=\im\left(\pi_1(\Omega_i)\longrightarrow\pi_1(M)\right)
            \end{displaymath}
            is virtually Abelian of rank $k_i$, with $0\leq k_i<n$.
    \end{enumerate}
    Then $\mu_L(\sing)=0$, and the geodesic flow is ergodic with respect to $\mu_L$.
    Further, if $k\coloneq\max_{i} k_{i}\geq 1$, then
    \begin{displaymath}
        \dim_H(\sing)\leq n+k-1\leq 2n-2.
    \end{displaymath}
    If $k_i=0$ for every $i$, then $\sing=\emptyset$ and the geodesic flow is Anosov.
\end{restatable}

Here, we call these compact regions $\Omega_i$ the exceptional regions because their union contains all points at which strict negativity of sectional curvature fails.

Finally, in the real-analytic setting, a characterization of the singular set in terms of curvature operators yields the following result:

\begin{restatable}{theorem}{analymlfd}\label{thm_2_6}
    Let $M$ be a closed, connected, real-analytic, rank-one manifold of nonpositive sectional curvature. Then its geodesic flow $\phi_t$ is ergodic with respect to the Liouville measure.
\end{restatable}

\section{Ergodicity on Smooth Surfaces}\label{sec3}
Throughout this section, we always assume $M$ is a closed surface of nonpositive curvature and genus at least $2$.

\begin{lemma}\label{lem_3_1}
    The set $F_{\infty}=\left\{x\in M\mid \nabla^{j}K(x)=0,~j\geq 0\right\}$ is a closed set.
\end{lemma}

\begin{proof}
    For every $j\ge0$, the function $x\mapsto \nabla^{j}K(x)$ is continuous. Hence,
    \begin{displaymath}
        F_{\infty}=\bigcap_{j=0}^{\infty}\{x\in M:\nabla^{j}K(x)=0\}
    \end{displaymath}
    is closed.
\end{proof}

\begin{lemma}\label{lem_3_2}
    For each $x\in M$, let
    \begin{displaymath}
        \Lambda_{x} = \left\{\bm{v}\in T^{1}_{x}M \mid K(c_{\bm{v}}(t))=0,~\forall t\in\mathbb{R}\right\}.
    \end{displaymath}

    If $x\notin F_{\infty}$, then $\Lambda_x$ is a finite set, thus has zero angular measure on $T_x^1M\simeq S^1$.
\end{lemma}

\begin{proof}
    Take any point $x\notin F_{\infty}$. If $K(x)<0$, $\Lambda_{x}=\emptyset$ by definition. If $K(x)=0$, then $K$ attains a local maximum at $x$ because $K\leq 0$. Hence $\mathrm{d} K(x)=0$. Since $x\notin F_{\infty}$, we can find the minimal $m\ge 2$, such that $\nabla^{m}K(x)\neq 0$.

    By definition, for any $\bm{v}\in \Lambda_{x}$, $K(c_{\bm{v}}(t))=0$ for any $t\in\mathbb{R}$. Thus,
    \begin{displaymath}
        \dfrac{\mathrm{d}^n}{\mathrm{d}t^{n}} K(c_{\bm{v}}(t))=0,\quad\forall n\in\mathbb{N}.
    \end{displaymath}

    In particular, $\dfrac{\mathrm{d}^{m}}{\mathrm{d}t^{m}} K(c_{\bm{v}}(t))=0$.

    We now claim that
    \begin{displaymath}
        \nabla^{n}K(x)(\overbrace{\bm{v},\ldots,\bm{v}}^{n})=\left.\dfrac{\mathrm{d}^{n}}{\mathrm{d}t^{n}}\right\rvert_{t=0}K(c_{\bm{v}}(t)).
    \end{displaymath}

    We prove this claim by induction. When $n=1$, by the chain rule
    \begin{displaymath}
        \left.\dfrac{\mathrm{d}}{\mathrm{d}t}\right\rvert_{t=0}K(c_{\bm{v}}(t))=\langle \operatorname{grad} K(c_{\bm{v}}(t)),c'_{\bm{v}}(t)\rangle\rvert_{t=0}=\langle \operatorname{grad} K(x),\bm{v}\rangle=\mathrm{d} K(x)(\bm{v}).
    \end{displaymath}

    Suppose the claim is true for $n=k$, then we can compute that
    \begin{displaymath}
        \begin{aligned}
            \dfrac{\mathrm{d}^{k+1}}{\mathrm{d}t^{k+1}}K(c_{\bm{v}}(t)) & = \dfrac{\mathrm{d}}{\mathrm{d}t}[\nabla^{k} K(c_{\bm{v}}(t))(\overbrace{c'_{\bm{v}}(t),\ldots,c'_{\bm{v}}(t)}^{k})]\\
                                                                        & =(\nabla_{c'_{\bm{v}}(t)}\nabla^{k}K(c_{\bm{v}}(t)))(\overbrace{c'_{\bm{v}}(t),\ldots,c'_{\bm{v}}(t)}^{k})\\
                                                                        & +\sum_{j=1}^k \nabla^{k}K(c_{\bm{v}}(t))(c'_{\bm{v}}(t),\ldots,\overbrace{\nabla_{c'_{\bm{v}}(t)}c'_{\bm{v}}(t)}^{j},\ldots, c'_{\bm{v}}(t)) \\
        \end{aligned}
    \end{displaymath}

    Given $c_{\bm{v}}(t)$ is a geodesic, $\nabla_{c'_{\bm{v}}(t)}c'_{\bm{v}}(t)=0$, terms in the summation all vanish. Therefore, we obtain the required formula:
    \begin{displaymath}
        \begin{aligned}
            \left.\dfrac{\mathrm{d}^{k+1}}{\mathrm{d}t^{k+1}}\right\rvert_{t=0}K(c_{\bm{v}}(t)) & =\nabla_{c'_{\bm{v}}(t)}\nabla^{k}K(c_{\bm{v}}(t))(\overbrace{c'_{\bm{v}}(t),\ldots,c'_{\bm{v}}(t)}^{k})\rvert_{t=0} \\
                                                                                                & =\nabla^{k+1}K(c_{\bm{v}}(t))(\overbrace{c'_{\bm{v}}(t),\ldots,c'_{\bm{v}}(t)}^{k+1})\rvert_{t=0}                    \\
                                                                                                & =\nabla^{k+1} K(x)(\overbrace{\bm{v},\ldots,\bm{v}}^{k+1}).
        \end{aligned}
    \end{displaymath}

    Now let $P_{m,x}(\bm{u})\coloneq(\nabla^{m}K(x))(\bm{u},\ldots,\bm{u})$, we can compute the coefficients of $P_{m,x}$ in local coordinates $(x_1,x_2)$ near the point $x$. Start with the first order derivative, we have
    \begin{displaymath}
        \nabla K(x)=\dfrac{\partial K(x)}{\partial x_{1}}\mathrm{d}x_{1} + \dfrac{\partial K(x)}{\partial x_{2}}\mathrm{d}x_2=0.
    \end{displaymath}

    Now the second order covariant derivative $\nabla^{2}K(x)$ is given by
    \begin{displaymath}
        \nabla_{j}\nabla_{i}K(x)=\dfrac{\partial^2 K(x)}{\partial x_{j}\partial x_{i}}-\Gamma^{k}_{ij}\dfrac{\partial K}{\partial x_{k}}=\dfrac{\partial^2 K(x)}{\partial x_{j}\partial x_{i}}.
    \end{displaymath}

    By induction on $l<m$, we obtain
    \begin{displaymath}
        0={(\nabla^{l} K(x))}_{i_{1}\cdots i_{l}}=\dfrac{\partial^{l} K(x)}{\partial x_{i_1}\cdots\partial x_{i_l}},\qquad i_1,\ldots,i_l=1,2,
    \end{displaymath}
    provided that all lower order derivatives vanish at $x$. In particular, when $l=m$
    \begin{displaymath}
        {\left(\nabla^{m} K(x)\right)}_{i_{1} \cdots i_{m}}=\dfrac{\partial^{m} K(x)}{\partial x_{i_{1}}\cdots\partial x_{i_{m}}}+\sum_{0\leq l\leq m-1} A\cdot B,\qquad i_1,\ldots,i_m=1,2.
    \end{displaymath}

    In the summation, $A$ stands for lower orders of partial derivatives of $K(x)$, which all vanish at $x$, and $B$ stands for the Christoffel symbols and their derivatives. Therefore, we get
    \begin{equation}\label{eq_3_1}
        {(\nabla^{m} K(x))}_{i_1\cdots i_m}=\dfrac{\partial^m K(x)}{\partial x_{i_1}\cdots\partial x_{i_m}}.
    \end{equation}

    Choose $\bm{u}\in T^{1}_x M$. In the local geodesic normal coordinate with the orthonormal frame $\{\bm{e}_{1}(y),\bm{e}_{2}(y)\}$ for $y$ near $x$, set $\bm{u}=u_{1}\,\bm{e}_{1}(x)+u_{2}\,\bm{e}_{2}(x)=u_{1}\,\partial_{x_1}+u_{2}\,\partial_{x_2}$, we can write
    \begin{equation}\label{eq_3_2}
        P_{m,x}(\bm{u})=\nabla^{m} K(x)(\overbrace{\bm{u},\ldots,\bm{u}}^{m})=\sum_{i_{1},\ldots,i_{m}=1}^{2}\dfrac{\partial^{m}K(x)}{\partial x_{i_1}\cdots\partial x_{i_m}}u_{i_1}\ldots u_{i_m}.
    \end{equation}

    Equation~\eqref{eq_3_2} shows that $P_{m,x}(\bm{u})$ is a nonzero homogeneous polynomial of degree $m$ in $u_1$ and $u_2$.

    Now, for any $\bm{v}\in \Lambda_x$, by definition, the geodesic $c_{\bm{v}}$ is flat, hence
    \begin{displaymath}
        P_{m,x}(\bm{v})=\dfrac{\mathrm{d}^{m}}{\mathrm{d}t^{m}}K(c_{\bm{v}}(t))=0.
    \end{displaymath}

    On $T^1_x M\simeq S^1$, let $\bm{v}=\cos\theta\bm{e}_1(x)+\sin\theta\bm{e}_2(x)$. By homogeneity, the restriction of $P_{m,x}(\bm{v})$ to the unit circle now becomes a  nonzero trigonometric polynomial of $\theta$ of degree at most $m$:

    \begin{equation}\label{eq_3_3}
        P_{m,x}(\bm{v}(x,\theta))=\sum_{i=0}^{m} a_{i}(x)\cos^{m-i}\theta\sin^{i}\theta,
    \end{equation}
    where $a_i(x)$ are smooth functions defined near $x$. For fixed $x$, there are only finitely many $\theta\in S^{1}$, such that $P_{m,x}=0$. Thus, $\Lambda_x$ is a finite set with $0$ angular measure.
\end{proof}

We now relate the projected singular set $\pi(\Lambda)\subset M$ to the infinite flat set $F_{\infty}$. $\Lambda$ consists of tangent vectors whose geodesics are flat, every point in $\pi(\Lambda)$ has zero Gaussian curvature. By contrast, $F_{\infty}$ consists of points at which $K$ vanishes to infinite order. Both $\pi(\Lambda)$ and $F_{\infty}$ are contained in the zero-curvature locus, but neither inclusion between $\pi(\Lambda)$ and $F_{\infty}$ is automatic. The following
proposition describes their relationship.

\begin{proposition}\label{prop_3_1}
    For $\mu_{L}$-almost every $\bm{v}\in\Lambda$, $c_{\bm{v}}(\mathbb{R})\subset F_{\infty}$.
\end{proposition}
\begin{proof}
    In local coordinates $\bm{v}=(x,\theta)$, the Liouville measure $\mu_L$ can be written as
    \begin{equation}
        \mathrm{d}\mu_L(\bm{v})=c\mathrm{d}A(x)\mathrm{d}\theta,
    \end{equation}
    where $\mathrm{d}A$ is the Riemannian area element on $M$, $\mathrm{d}\theta$ is the angular measure on the fiber $T^{1}_x M\simeq S^1$, and $c=\dfrac{1}{2\pi \area(M)}$ is the normalization constant. From Lemma~\ref{lem_3_2}, we know that
    \begin{displaymath}
        \theta_{x}(\Lambda_{x})=0,\quad \forall x\notin F_{\infty}.
    \end{displaymath}

    Thus, we can conclude that
    \begin{displaymath}
        \mu_{L}(\left\{\bm{v}\in\Lambda\mid c_{\bm{v}}(0)\notin F_{\infty}\right\})=c\int_{\pi(\Lambda)\setminus F_{\infty}}\theta_{x}(\Lambda_{x})\,\mathrm{d}A(x)=0.
    \end{displaymath}

    Here $\Lambda$ is compact, then $\pi(\Lambda)$ is also compact and hence measurable. Note that both $\mu_L$ and $\Lambda$ are invariant under the geodesic flow $\phi_t$, for any $q\in\mathbb{Q}$, we obtain that:
    \begin{displaymath}
        \begin{aligned}
            \mu_L\left(\left\{\bm{v}\in\Lambda\mid c_{\bm{v}}(q)\notin F_{\infty}\right\}\right) & =\mu_L\left(\left\{\bm{v}\in\Lambda\mid c_{\phi_{q}\bm{v}}(0)\notin F_{\infty}\right\}\right)    \\
                                                                                                 & =\mu_{L}\left(\phi_{-q}\left\{\bm{v}\in\Lambda\mid c_{\bm{v}}(0)\notin F_{\infty}\right\}\right) \\
                                                                                                 & =0.
        \end{aligned}
    \end{displaymath}

    By the countable subadditivity of measure, we have
    \begin{displaymath}
        \mu_L\left(\bigcup_{q\in\mathbb{Q}}\left\{\bm{v}\in\Lambda\mid c_{\bm{v}}(q)\notin F_{\infty}\right\}\right)=0.
    \end{displaymath}

    Let $N$ denote this union, which is a set of zero Liouville measure in $\Lambda$. Thus, for any $\bm{u}\in\Lambda\setminus N$, for any $q\in\mathbb{Q}$, we always have $c_{\bm{u}}(q)\in F_{\infty}$.

    From Lemma~\ref{lem_3_1}, $F_{\infty}$ is a closed set. Therefore, $c_{\bm{u}}(\mathbb{R})\subset F_{\infty}$ as required.
\end{proof}

\begin{remark}
    If $\mu_L(\Lambda)>0$, Proposition~\ref{prop_3_1} and Fubini's theorem imply the existence of a measurable set $E\subset F_{\infty}$ of positive area such that, for every $x\in E$, a set of initial directions of positive angular measure generates geodesics that remain entirely in $F_{\infty}$.
\end{remark}

Further, as to the Hausdorff dimension, we have the following result:

\begin{proposition}[Hausdorff Dimension]\label{prop_3_2}
    \begin{displaymath}
        \dim_{H}\Lambda \leq \max\left\{2, 1+\dim_{H}F_{\infty}\right\}
    \end{displaymath}
\end{proposition}

\begin{proof}
    Set $Z(K)=\{x\in M\mid K(x)=0\}$ and we know $\Lambda\subset\pi^{-1}(Z(K))$. We first show that $Z(K)\setminus F_{\infty}$ is covered by at most countably many smooth curves on $M$.

    Choose a countable smooth atlas $(U_{j},\psi_{j})$ covering $M$. In each $U_{j}$ and each multi-index $\beta\in\mathbb{N}^{2}$, define that
    \begin{displaymath}
        H_{j,\beta}=\{y\in U_{j}\mid \partial^{\beta}K(y)=0, {\mathrm{d}(\partial^{\beta}K)}_{y}\neq 0\}.
    \end{displaymath}

    For each $p\in H_{j,\beta}$, the Implicit Function Theorem suggests there exists an open neighborhood $V_{p}\subset U_{j}$, such that
    \begin{displaymath}
        H_{j,\beta}\cap V_{p}={(\partial^{\beta}K)}^{-1}(0)\cap V_{p}
    \end{displaymath}
    is a piece of smooth curve in $U_j$. Given $U_{j}$ is second countable, for all $p\in H_{j,\beta}$, the relative open cover ${\{H_{j,\beta}\cap V_{p}\}}_{p}$ has a countable sub-cover.

    Now for any $x\in Z(K)\setminus F_{\infty}$ and $x\in U_{j}$ for some $j$, since $x$ is not in $F_{\infty}$, we can find a multi-index $\alpha$ with minimal total order, such that
    \begin{displaymath}
        \partial^{\alpha} K(x)\neq 0.
    \end{displaymath}

    Note that $K(x)=0$, we have $|\alpha|\ge 2$. Choose $i\in\{1,2\}$ with $\alpha_{i}>0$ and let $\beta=\alpha-\bm{e}_{i}$, we have
    \begin{displaymath}
        \partial^{\beta} K(x)=0,\qquad \partial^{\alpha} K(x)\neq 0,\qquad x\in H_{j,\beta}.
    \end{displaymath}

    Thus, $Z(K)\setminus F_{\infty}\subset \bigcup_{j,\beta} H_{j,\beta}$. Each $H_{j,\beta}$ is covered by countably many smooth curves, $j$ and $\beta$ are countable indices, hence $Z(K)\setminus F_{\infty}$ is covered by countably many smooth curves, denoted by $\{l_{i}\}$, with dimension at most $1$.

    Note that the projection $\pi: T^{1}M\to M$ is a smooth submersion, for each $i$, $\pi^{-1}(l_{i})$ is a smooth embedded two-dimensional submanifold in $T^{1}M$. Hence, we get that
    \begin{displaymath}
        \Lambda\setminus\pi^{-1}(F_{\infty})\subset \pi^{-1}(Z(K)\setminus F_{\infty})
    \end{displaymath}
    is covered by countably many smooth two-dimensional submanifolds, whose union has Hausdorff dimension at most $2$.

    For the remaining part, consider the dimension of $\pi^{-1}(F_{\infty})$. Using smooth local trivializations of the unit tangent bundle on relatively compact coordinate neighborhoods and the product estimate with the smooth factor \(S^1\), we obtain

    \begin{displaymath}
        \dim_{H}\pi^{-1}(F_{\infty})\le \dim_{H}F_{\infty}+1.
    \end{displaymath}

    Finally, note that
    \begin{displaymath}
        \Lambda\subset \left(\Lambda\setminus\pi^{-1}(F_{\infty})\right)\bigcup \pi^{-1}(F_{\infty}),
    \end{displaymath}
    and the Hausdorff dimension is controlled by
    \begin{displaymath}
        \dim_{H}\Lambda\leq\max\left\{2,\dim_{H}(F_{\infty})+1\right\}.
    \end{displaymath}
\end{proof}

In particular, $\dim_{H}F_{\infty}<2$ implies $\dim_{H}\Lambda<3$, and hence $\mu_L(\Lambda)=0$. The condition $\area(F_\infty)=0$, however, does not imply $\dim_{H}F_{\infty}<2$.

Now we are ready to prove the first main theorem.

\surface*

\begin{proof}
    By Proposition~\ref{prop_3_1},

    \begin{displaymath}
        \begin{aligned}
            \mu_L(\Lambda) & =\mu_L(\Lambda\cap\pi^{-1}(F_\infty))                            \\
                           & \le \mu_L(\pi^{-1}(F_{\infty}))                                \\
                           & =\int_{F_{\infty}} \frac{1}{2\pi \area(M)} 2\pi \,\mathrm{d}A(x) \\
                           & =\frac{\area(F_\infty)}{\area(M)}=0.
        \end{aligned}
    \end{displaymath}

    Therefore, $\mu_{L}(\Lambda)=0$, Theorem~\ref{thm_2_2} gives the ergodicity.
\end{proof}

Next we prove Corollary~\ref{cor_2_1}.

\analysurface*
\begin{proof}
    Assume $F_{\infty}\neq\emptyset$, for any $x\in F_{\infty}$, derivatives of all orders of $K(x)$ vanish. Note that $K:M\to\mathbb{R}$ is now a real-analytic function because the metric is real analytic. By the connectedness and analytic continuation, $K\equiv0$ on $M$.

    On the other hand, Gauss-Bonnet Theorem requires
    \begin{equation}
        0=\int_{M} K(x)\,\mathrm{d}A(x)=2\pi\chi(M)<0.
    \end{equation}

    Contradiction. Therefore, Theorem~\ref{thm_2_3} applies in the case $F_{\infty}$ is an empty set.
\end{proof}

\section{An Example with Infinitely Many Negative-Curvature Components}\label{sec4}

On a smooth, connected and closed surface of nonpositive curvature and genus at least $2$, suppose the set of points with negative curvature $N=\{x\in M\mid K(x)<0\}$ has only finitely many connected components, then the geodesic flow is ergodic with respect to the Liouville measure~\cite{Wu2015,WuLiuWang2023}.

This assumption is independent of our condition, $\area(F_{\infty})=0$. In fact, the following example shows that the condition $\area(F_{\infty})=0$ does not require the negative-curvature part to have finitely many connected components.

\begin{example}
    There exists a closed, connected surface of nonpositive curvature and genus $2$, such that:
    \begin{enumerate}
        \item $\area(F_{\infty})=0$.
        \item The set of negative curvature points $N$ has infinitely many connected components.
        \item The geodesic flow is ergodic but not Anosov.
    \end{enumerate}
\end{example}

We start with a closed hyperbolic surface $(M,g_0)$ of genus $2$, where $g_{0}$ is a smooth hyperbolic metric with constant Gaussian curvature $K_{g_0}\equiv -1$. Let the surface admit an orientation-reversing isometric involution $\iota$ whose fixed-point set is a separating simple closed geodesic $c_{0}$. $M\setminus c_{0}$ has two connected components, each of them is a hyperbolic one-holed torus (without the boundary) and $\iota$ exchanges them.

Equivalently, we can construct $(M,g_0)$ as the double of a hyperbolic one-holed torus of constant Gaussian curvature $-1$ with totally geodesic boundary $c_{0}$.
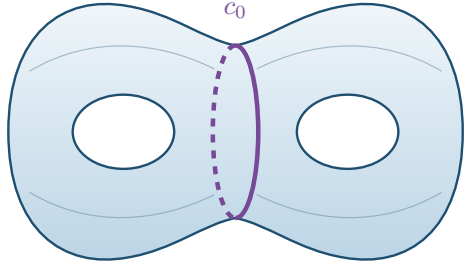
\begin{figure}[htbp]
    \centering
    \begin{tikzpicture}
        \begin{scope}[scale=1]
            \path[surface]
            (-3.00,0).. controls (-3.00,1.30) and (-2.28,1.82).. (-1.42,1.66).. controls (-0.82,1.56) and (-0.42,1.22).. (0,1.14).. controls (0.42,1.22) and (0.82,1.56).. (1.42,1.66).. controls (2.28,1.82) and (3.00,1.30).. (3.00,0).. controls (3.00,-1.30) and (2.28,-1.82).. (1.42,-1.66).. controls (0.82,-1.56) and (0.42,-1.22).. (0,-1.14).. controls (-0.42,-1.22) and (-0.82,-1.56).. (-1.42,-1.66).. controls (-2.28,-1.82) and (-3.00,-1.30).. (-3.00,0)-- cycle;
            \path[fill=white, draw=surfaceedge, line width=0.9pt] (-1.48,0) ellipse (0.67 and 0.49);
            \path[fill=white, draw=surfaceedge, line width=0.9pt] (1.48,0) ellipse (0.67 and 0.49);
            \draw[contour] (-2.72,0.80)..controls (-1.92,1.28) and (-1.02,1.20).. (-0.28,0.83);
            \draw[contour] (2.72,0.80)..controls (1.92,1.28) and (1.02,1.20).. (0.28,0.83);
            \draw[contour] (-2.72,-0.80)..controls (-1.92,-1.28) and (-1.02,-1.20).. (-0.28,-0.83);
            \draw[contour] (2.72,-0.80)..controls (1.92,-1.28) and (1.02,-1.20).. ( 0.28,-0.83);
            \draw[seam, dashed] (0,1.14) arc[start angle=90,end angle=270,x radius=0.30,y radius=1.14];
            \draw[seam] (0,-1.14) arc[start angle=-90,end angle=90,x radius=0.30,y radius=1.14];
        \end{scope}

        \node[font=\small, text=glueboundary, fill=white, inner sep=1.5pt] at (0,1.6) {$c_{0}$};
        \node[font=\small] at (0,-2.53) {$K_{g_{0}}\equiv -1$\qquad $\operatorname{Fix}(\iota)=c_{0}$\qquad $\iota^{\ast}g_{0}=g_{0}$};

        \node[note] at (0.35,-3.3) {$c_0$ is the separating closed geodesic.};
    \end{tikzpicture}
    \caption{Genus Two Surface with Constant Negative Curvature}
\end{figure}

Now choose a sufficiently small, $\iota$-invariant, tubular neighborhood $U$ containing $c_0$ and the local coordinates $U\simeq (-3a,3a)\times S^{1}$ where $a$ is some small constant. Assign the coordinate $(r,\theta)$ to the point in this trivialization with $\iota(r,\theta)=(-r,\theta)$. In particular, we have $c_0=\{(0,\theta)\mid \theta\in S^{1}\}$ in this coordinate.

Define the function $s:\mathbb{R}\to\mathbb{R}$ by
\begin{equation}
    s(t)=\left\{
        \begin{aligned}
         & \mathrm{e}^{-\frac{1}{t}},\qquad & t>0,     \\
         & 0,\qquad                         & t\leq 0.
        \end{aligned}
    \right.
\end{equation}

This function $s$ is smooth with $s^{(j)}(0)=0$ for all $j\ge 0$.

Now we can define the cutoff function
\begin{equation}
    \begin{aligned}
        \eta:(-3a,3a) & \longrightarrow \mathbb{R} &                                                                  \\
        r             & \longmapsto \eta(r)        & \coloneq \dfrac{s(4a^2-r^2)}{s(4a^2-r^2)+s(r^2-\frac{9a^2}{4})}.
    \end{aligned}
\end{equation}

It is easy to check that this cutoff function $\eta$ is smooth, constant $1$ for $|r|\le \frac{3a}{2}$ and $0$ for $|r|\geq 2a$. It takes values in $[0,1]$, and all derivatives of positive order vanish at the transition endpoints.

Lastly, define an even function $H:(-3a,3a)\to \mathbb{R}$ by
\begin{equation}
    H(r)=\left\{
        \begin{aligned}
         & \mathrm{e}^{-\frac{a^2}{r^2}}\sin^{2}\left(\frac{a\pi}{r}\right),\qquad & r\neq 0, \\
         & 0,\qquad                                                                & r=0.
        \end{aligned}
    \right.
\end{equation}

Clearly, $H\geq 0$. And $H(r)=0$ if and only if $r=0$ or $r=\frac{a}{n}$ for integer $n\neq 0$. Denote the zeros of $H$ by
\begin{displaymath}
    Z(H)=H^{-1}(0)=\{0\}\bigcup_{n=1}^{\infty}\{\frac{a}{n},-\frac{a}{n}\}.
\end{displaymath}

Next, we can compute the derivatives. When $r=0$, we have $H^{(n)}(0)=0$ for every integer $n\ge 0$, $H(r)$ is infinitely flat at $0$. For the other zeros $r_n=\frac{a}{n}$, we have
\begin{displaymath}
    H(r_n)=H'(r_n)=0;\qquad H''(r_n)=\frac{2\pi^2}{a^2} n^4\mathrm{e}^{-n^2}>0.
\end{displaymath}

Thus, $H$ is not infinitely flat at $r_n$ or $-r_n$.

Lastly, define a function $q:M\to \mathbb{R}$ by
\begin{displaymath}
    q(x)=\left\{
        \begin{aligned}
         & \eta(r)H(r)+(1-\eta(r)),\qquad &  & x\in U,            \\
         & 1,\qquad                       &  & x\in M\setminus U.
        \end{aligned}
    \right.
\end{displaymath}

It is easy to check that $q(x)$ is a smooth function on $M$ because near the boundary of the tubular neighborhood $q\equiv1$. Moreover, we have $q\circ\iota=q$. We can compute that for $x=(r,\theta)\in U$,
\begin{displaymath}
    q(x)=q(r,\theta)=\left\{
        \begin{aligned}
         & H(r),\qquad              &  & |r|\leq\frac{3a}{2},         \\
         & 1+\eta(r)(H(r)-1),\qquad &  & \frac{3a}{2}\leq |r|\leq 2a, \\
         & 1,\qquad                 &  & |r|\geq 2a.
        \end{aligned}
    \right.
\end{displaymath}

The set of zeros of $q(x)$ is given by circles:
\begin{displaymath}
    Z(q)=q^{-1}(0)=c_{0}\bigcup_{j=1}^{\infty}(c_{j}^{+}\cup c_{j}^{-}),\qquad c_{j}^{\pm}=\{(\pm\frac{a}{j},\theta)\mid \theta\in S^{1}\}.
\end{displaymath}

Note that $q$ is not constantly $0$ on $M$, by Kazdan-Warner Theorem (see~\cite{KazdanWarner1974}), there is a unique metric $g$ conformal to $g_0$ on $M$ such that the Gaussian curvature $K_{g}(x)=-q(x)$. In fact, there exists a unique smooth function $u$, such that
\begin{equation}
    \Delta_{g_0}u-K_{g_{0}}=\Delta_{g_0}u+1=q\mathrm{e}^{2u}.
\end{equation}

Let $g=\mathrm{e}^{2u}g_{0}$, we can compute
\begin{displaymath}
    K_{g}=\mathrm{e}^{-2u}(K_{g_{0}}-\Delta_{g_{0}}u)=e^{-2u}(-1-\Delta_{g_{0}}u)=-q.
\end{displaymath}

Since $\iota$ preserves both $g_{0}$ and $q$, the function $u\circ\iota$ satisfies the same equation as $u$. Uniqueness gives $u\circ\iota=u$, and hence $\iota^{\ast}g=g$. Therefore, the fixed-point set $c_0$ is a geodesic with respect to $g$. Since $K_{g}|_{c_0}=0$, it is a closed flat geodesic.

By Theorem~\ref{thm_2_1}, the geodesic flow of $g$ is not Anosov.

Next, we will verify $\area(F_{\infty}(g))=0$. By definition, we have
\begin{displaymath}
    F_{\infty}(g)=\{x\mid \nabla_{g}^{j}q(x)=0,\quad\forall j\ge 0\}\subset Z(q).
\end{displaymath}

Along $c_0$, all partial derivatives of $q$ vanish, so $c_0\subset F_\infty(g)$. But the other circles $c_{n}^{\pm}$, we have
\begin{displaymath}
    K_{g}(x)\vert_{x\in c_{n}^{\pm}}=-H(\pm\frac{a}{n})=0,\qquad \partial^2_{r}H(\pm\frac{a}{n})\neq 0.
\end{displaymath}

Therefore, we can conclude that $F_{\infty}(g)=c_0$, which has zero area. Thus, the geodesic flow is ergodic with respect to the Liouville measure.

Lastly, each annulus $(a/(n+1),a/n)\times S^1$ and its reflected counterpart is contained in $N$ and is bounded by zero-curvature circles. These annuli are distinct connected components of $N$. Hence, $N$ has infinitely many connected components.

\section{Ergodicity on Smooth Manifolds in Higher Dimensions}\label{sec5}
Throughout this section, we will work on a smooth closed rank-one manifold $(M,g)$ of nonpositive curvature. Let $\mu_L$ denote the normalized Liouville measure with $\mu_L(T^{1}M)=1$ and $\vol$ denote the Riemannian volume of $M$. For any $x\in M$, let $\sigma_{x}$ be the smooth normalized rotationally invariant measure on the sphere $T^{1}_{x}M\simeq S^{n-1}$, so that $\sigma_{x}(T^{1}_{x}M)=1$. Choose $\bm{v}\in T^{1}_{x}M$, define the orthogonal complement of $\bm{v}$ as
\begin{displaymath}
    E_{\bm{v}}\coloneq\bm{v}^{\perp}=\left\{\bm{w}\in T_{x}M\mid \langle\bm{v},\bm{w}\rangle=0\right\}.
\end{displaymath}

For any $t\in\mathbb{R}$, let
\begin{displaymath}
    P_{\bm{v},t}:E_{\bm{v}}\longrightarrow E_{\phi_{t}(\bm{v})}
\end{displaymath}
be the parallel transport along the geodesic $c_{\bm{v}}$. Since the parallel transport along a fixed geodesic preserves the Riemannian inner product, this map is well-defined and a linear isometry.

By definition, the singular set can be written as
\begin{displaymath}
    \begin{aligned}
        \sing & =\left\{\bm{v}\in T^{1}M\mid \rank(\bm{v})\ge2\right\}                                                                                                         \\
              & =\left\{\bm{v}\in T^{1}M\mid \exists 0\neq \bm{w}\in E_{\bm{v}},~R(P_{\bm{v},t}(\bm{w}),c'_{\bm{v}}(t))c'_{\bm{v}}(t)=0,\quad \forall t\in\mathbb{R} \right\}.
    \end{aligned}
\end{displaymath}

By the Pesin-Hopf argument, the geodesic flow is ergodic on the regular rank-one set; therefore, if $\mu_L(\sing)=0$, it is ergodic on $T^{1}M$ with respect to the Liouville measure. See~\cite{BallmannBrin1982,Burns1983,BurnsMatveev2021, Pesin1977}. Our target is to prove the singular set has zero Liouville measure. Start with the curvature Jacobi operator:

\begin{definition}[Curvature Jacobi Operator]\label{def_5_1}
    For any $x\in M$ and $\bm{0}\neq\bm{v}\in T_{x}M$, let $J_{\bm{v}}:T_{x}M \to T_{x}M$ denote the Curvature Jacobi Operator defined by
    \begin{displaymath}
        J_{\bm{v}}(\bm{w})=-R(\bm{w},\bm{v})\bm{v}.
    \end{displaymath}
\end{definition}

Restricting the operator to $E_{\bm{v}}$ for the given $\bm{v}\neq 0$, we can also define the Reduced Jacobi Operator as

\begin{definition}[Reduced Jacobi Operator]\label{def_5_2}
    For any $x\in M$ and $\bm{v}\in T^{1}_{x}M$, we define the Reduced Normal Jacobi Operator as
    \begin{displaymath}
        \begin{aligned}
            J^{\perp}_{\bm{v}}: & E_{\bm{v}} & \longrightarrow & ~~E_{\bm{v}}             \\
                                & \bm{w}     & \longmapsto     & -R(\bm{w},\bm{v})\bm{v}.
        \end{aligned}
    \end{displaymath}
\end{definition}

For convenience, our definitions~\ref{def_5_1} and~\ref{def_5_2} differ from the classical ones by a sign. More discussions of these operators can be found in~\cite{BlazicGilkeyNikcevicSimon2006,Gilkey2001}.

\begin{proposition}
    The curvature Jacobi operator and reduced Jacobi operator have the following properties:
    \begin{enumerate}
        \item $J_{\bm{v}}$ is linear and self-adjoint, so is $J^{\perp}_{\bm{v}}$.
        \item $J_{\bm{v}}(\bm{v})=0$. $\bm{v}$ is an eigenvector to the eigenvalue $0$.
        \item $J_{\bm{v}}(E_{\bm{v}})\subset E_{\bm{v}}$, thus, the reduced Jacobi operator is well-defined.
        \item $J^{\perp}_{\bm{v}}$ is positive semidefinite if $(M,g)$ has nonpositive sectional curvature.
    \end{enumerate}
\end{proposition}

\begin{proof}
    Linearity follows from the Riemannian curvature tensor. For any $\bm{u},\bm{w}\in T_{x}M$, we can compute
    \begin{displaymath}
        \langle J_{\bm{v}}(\bm{w}),\bm{u}\rangle=-\langle R(\bm{w},\bm{v})\bm{v},\bm{u}\rangle=-\langle R(\bm{u},\bm{v})\bm{v},\bm{w}\rangle= \langle J_{\bm{v}}(\bm{u}),\bm{w}\rangle.
    \end{displaymath}

    Hence, $J^{\ast}_{\bm{v}}=J_{\bm{v}}$. $J_{\bm{v}}(\bm{v})=0$ is clear from the definition.

    For any $\bm{w}\in E_{\bm{v}}$, we have
    \begin{displaymath}
        \langle J_{\bm{v}}(\bm{w}),\bm{v}\rangle=-\langle R(\bm{w},\bm{v})\bm{v},\bm{v}\rangle=0.
    \end{displaymath}

    Thus, $J_{\bm{v}}(\bm{w})\in E_{\bm{v}}$. The restriction $J^{\perp}$ is well-defined.

    Lastly, for any unit vector $\bm{w}\in E_{\bm{v}}$, we have
    \begin{displaymath}
        \langle J^{\perp}_{\bm{v}}(\bm{w}),\bm{w}\rangle=-\langle R(\bm{w},\bm{v})\bm{v},\bm{w}\rangle=-K(\bm{v},\bm{w}).
    \end{displaymath}

    This is the negative of the sectional curvature, hence, it is positive semidefinite given $(M,g)$ has nonpositive sectional curvature.
\end{proof}

\begin{definition}[Reduced Jacobi Curvature Determinant]\label{def_5_3}
    For given $x$ and $\bm{v}\in T^{1}_{x}M$, we denote the determinant of $J^{\perp}_{\bm{v}}$ by
    \begin{displaymath}
        \begin{aligned}
            q_{J}: & T^{1}M     & \longrightarrow & \mathbb{R}                \\
                   & (x,\bm{v}) & \longmapsto     & \det(J^{\perp}_{\bm{v}}).
        \end{aligned}
    \end{displaymath}

    On manifolds of nonpositive curvature, we can see that $q_{J}\geq 0$.
\end{definition}

Note that $J^{\perp}_{\bm{v}}$ is self-adjoint, we can choose an orthonormal eigenbasis $\{\bm{e}_1,\ldots,\bm{e}_{n-1}\}$ of $E_{\bm{v}}$. Denote the sectional curvature by
\begin{displaymath}
    \kappa_{i}=-\langle J^{\perp}_{\bm{v}}(\bm{e}_i),\bm{e}_i\rangle=\langle R(\bm{e}_i,\bm{v})\bm{v},\bm{e}_i\rangle,
\end{displaymath}
which is the sectional curvature of $\bm{v}$ and $\bm{e}_i$. We can compute
\begin{equation}
    q_J(x,\bm{v})=\prod_{i=1}^{n-1}(-\kappa_{i}).
\end{equation}

By definition, we have
\begin{displaymath}
    q_{J}(x,\bm{v})=0 \Leftrightarrow\ker J^{\perp}_{\bm{v}}\neq \{0\}\Leftrightarrow \exists \bm{0}\neq\bm{w}\in E_{\bm{v}},~J^{\perp}_{\bm{v}}(\bm{w})=\bm{0}.
\end{displaymath}

\begin{lemma}\label{lem_5_1}
    Fix $x\in M$ and a local frame near $x$, denote $q_{J}^{x}=q_J(x,\cdot):T^{1}_{x}M\to \mathbb{R}$.
    \begin{displaymath}
        q_{J}^{x}(\bm{v})=\det(J_{\bm{v}}^{\perp}).
    \end{displaymath}

    Then $q_{J}^{x}$ is the restriction to $T^{1}_{x}M\simeq S^{n-1}$ of a homogeneous polynomial of degree $2(n-1)$ in the components of $\bm{v}$.
\end{lemma}

\begin{proof}
    Choose a smooth orthonormal frame $\{\bm{e}_1(y),\ldots\bm{e}_n(y)\}$ for $y\in U$ in a small neighborhood $U$ of $x$. In $T_{x}M$, we use the basis $\{\bm{e}_1(x),\dots,\bm{e}_{n}(x)\}$ and write
    \begin{displaymath}
        J_{\bm{v}}={(a_{ij}(x,\bm{v}))}_{n\times n}.
    \end{displaymath}
    This is an $n\times n$ matrix with elements $a_{ij}(x,\bm{v})$ being smooth functions of $\bm{v}$. Write $\bm{v}=\sum_{k=1}^n v_k\bm{e}_{k}(x)$, we have

    \begin{displaymath}
        \begin{aligned}
            a_{ij}(x,\bm{v}) & =\langle J_{\bm{v}}(\bm{e}_j(x)),\bm{e}_{i}(x) \rangle                                                        \\
                             & =-\langle R(\bm{e}_{j}(x),\bm{v})\bm{v},\bm{e}_{i}(x)\rangle                                                  \\
                             & =-\langle R(\bm{e}_{j}(x),\sum_{k=1}^n v_k\bm{e}_{k}(x))(\sum_{l=1}^{n}v_l\bm{e}_{l}(x)),\bm{e}_{i}(x)\rangle \\
                             & =-\sum_{k,l=1}^{n} v_k v_l\langle R(\bm{e}_{j}(x),\bm{e}_{k}(x))\bm{e}_{l}(x),\bm{e}_{i}(x)\rangle.
        \end{aligned}
    \end{displaymath}

    Once we fix $x$ and the frame, the curvature components are independent of $\bm{v}$. We can see that $a_{ij}(x,\bm{v})$ is a homogeneous polynomial of degree $2$ in terms of $v_k$'s.

    Claim that
    \begin{equation}
        q_{J}^{x}(\bm{v})=\det (J_{\bm{v}}^{\perp})=\operatorname{tr}(\bigwedge^{n-1}J_{\bm{v}})=\sigma_{n-1}(J_{\bm{v}}).
    \end{equation}

    Here $\sigma_{n-1}(J_{\bm{v}})$ is the $(n-1)$-th elementary symmetric polynomial defined in
    \begin{displaymath}
        \det(J_{\bm{v}}+t I)=t^n+\sigma_1(J_{\bm{v}})t^{n-1}+\cdots+\sigma_{n-1}(J_{\bm{v}})t+\det(J_{\bm{v}}).
    \end{displaymath}

    Note that $J_{\bm{v}}(\bm{v})=0$, we have the invariant decomposition $J_{\bm{v}}=0\oplus J^{\perp}_{\bm{v}}$ on $T_{x}M=\mathbb{R}\bm{v}\oplus E_{\bm{v}}$, thus, in $\sigma_{n-1}(J_{\bm{v}})$, all terms vanish except the product of the other $n-1$ eigenvalues of $J_{\bm{v}}$, which is exactly our $\det(J^{\perp}_{\bm{v}})$. Because each element in the matrix is a homogeneous polynomial of degree $2$ in terms of $v_{k}$'s, the corresponding $\sigma_{n-1}(J_{\bm{v}})$ is a homogeneous polynomial of degree $2(n-1)$ restricted to the sphere.
\end{proof}

\begin{definition}[Fiberwise Second Moment of the Reduced Jacobi Determinant]
    Define a function $h_J:M\to\mathbb{R}$ by
    \begin{displaymath}
        h_{J}(x)=\int_{T^{1}_{x}M}q_{J}^{2}(x,\bm{v})\,\mathrm{d}\sigma_{x}(\bm{v}).
    \end{displaymath}

    Clearly, $h_{J}$ is smooth and non-negative. $h_{J}(x)=0$ if and only if $q_{J}(x,\cdot)$ vanishes identically on $T^{1}_{x}M$.
\end{definition}

Now we will define the analogous set, denoted by $F^{J}_{\infty}$ to the infinite flat set $F_{\infty}$ (Definition~\ref{def_2_1}) as in the case of dimension $2$.

\begin{definition}
    On $M$ we define two sets:
    \begin{enumerate}
        \item The set of degenerate points $\mathcal{D}\coloneq h_{J}^{-1}(0)=\{x\in M\mid q_{J}^{x}\equiv 0\}$.
        \item The Infinite-order Vanishing Set of $h_{J}$:
            \begin{displaymath}
                F_{\infty}^{J}\coloneq\{x\in M\mid {(\nabla^{k}h_{J})}_{x}=0,\quad \forall k=0,1,\ldots\}.
            \end{displaymath}
    \end{enumerate}
\end{definition}

As we computed in the proof of Lemma~\ref{lem_3_2}, for point $x\in F^{J}_{\infty}$, one can check for any multi-index $\alpha\in\mathbb{N}^{n}$, $\partial^{\alpha}h_{J}(x)=0$ for any $|\alpha|\ge 0$ in any local chart.

This definition coincides with $F_{\infty}$ on surfaces, on which $q_{J}(x,\bm{v})=-K(x)$ and thus, $h_{J}(x)={K(x)}^2$. Consequently, $F^{J}_{\infty}=F_{\infty}$.

\begin{lemma}\label{lem_5_2}
    The set of zeros of finite-order of the smooth function $h_{J}$, denoted by $\mathcal{D}\setminus F^{J}_{\infty}$ can be covered by countably many hypersurfaces of dimension $n-1$. Thus, $\vol(\mathcal{D}\setminus F^{J}_{\infty})=0$.
\end{lemma}

\begin{proof}
    Choose a countable coordinate atlas $(U_{j},\psi_{j})$ covering $M$. In each chart $U_{j}$ and for each multi-index $\beta$, consider
    \begin{displaymath}
        H_{j,\beta}=\left\{y\in U_{j}\mid \partial^\beta h_J(y)=0,\; d{(\partial^\beta h_J)}_y\ne0\right\}.
    \end{displaymath}
    By the Implicit Function Theorem, for any point $p\in H_{j,\beta}$, we can find a sufficiently small neighborhood $V_{p}\subset U_{j}$ of it, such that
    \begin{displaymath}
        H_{j,\beta}\cap V_{p}={(\partial^{\beta}h_{J})}^{-1}(0) \cap V_{p}
    \end{displaymath}
    is a smooth hypersurface. Note that $U_j$ is second countable, hence the relative open cover $H_{j,\beta}\cap V_{p}$ for all $p\in H_{j,\beta}$ admits a countable open subcover.

    Now in some $U_{j}$, choose $x\in \mathcal{D}\setminus F_\infty^J$, choose a derivative $\partial^{\alpha} h_J(x)\ne0$ of minimal order. Select $i$ with $\alpha_{i}>0$ and set $\beta=\alpha-e_{i}$. Then
    \begin{displaymath}
        \partial^\beta h_J(x)=0,\qquad \partial_{i}\partial^{\beta} h_J(x)\ne0.
    \end{displaymath}

    Thus, $x\in H_{j,\beta}$. Therefore, this set $\mathcal{D}\setminus F^{J}_{\infty}$ is contained in the union of all possible $H_{j,\beta}$, each of which is covered by a collection of countably many hypersurfaces. Since $j$ and the multi-indices $\beta$ are countable, the countable cover of $\mathcal{D}\setminus F^{J}_{\infty}$ follows, which implies
    \begin{displaymath}
        \vol(\mathcal{D}\setminus F^{J}_{\infty})=0.
    \end{displaymath}

    Equivalently, we can see that $\vol(\mathcal{D})=\vol(F^{J}_{\infty})$.
\end{proof}

\begin{remark}
    This result actually holds for any smooth function $f(x)$ defined on a smooth manifold of dimension $n$. On surfaces, we have $\area(F_{\infty})=\area (\{x\mid K(x)=0\})$.
\end{remark}

\begin{lemma}\label{lem_5_3}
    If $x\notin \mathcal{D}$, we have $q_J(x,\cdot):T^{1}_{x}M\to\mathbb{R}$ is not constantly zero, and
    \begin{displaymath}
        \sigma_{x}\left( \left\{\bm{v}\in T^{1}_{x}M\mid q_{J}(x,\bm{v})=0\right\} \right)=0.
    \end{displaymath}
\end{lemma}
The following proof uses the polynomial zero-set argument used in Lemma~\ref{lem_3_2} in higher-dimensional fibers.

\begin{proof}
    Fix $x$ and some local chart $(U,\psi)$. By Lemma~\ref{lem_5_1}, $q_{J}^{x}$ is the restriction to $S^{n-1}$ of a nonzero homogeneous polynomial of degree $2(n-1)$ with constant coefficients depending only on $x$. Denote the polynomial by $P$ and the zero set of $P$ in $T_{x}M\simeq \mathbb{R}^{n}$ by $Z(P)$.

    By the homogeneity of $P$, if a point $\omega\in S^{n-1}$ is a zero point of $q_{J}^{x}$, then the entire line passing through the origin and $\omega$ is also contained in $Z(P)$. Using the polar coordinates and the Fubini's theorem, we have the Lebesgue measure of $Z(P)$ in $\mathbb{R}^n$ can be computed by
    \begin{displaymath}
        0=\operatorname{Leb}(Z(P))=\omega_{n-1}\int_0^{\infty} t^{n-1}\sigma_x(Z(P)\cap S^{n-1})\,\mathrm{d}t.
    \end{displaymath}
    Here $\omega_{n-1}$ is the volume of the sphere $S^{n-1}$. It is a standard fact that the zero set of a nonzero multivariate polynomial on $\mathbb{R}^{n}$ has Lebesgue measure zero. Therefore, the zero set of $q_{J}(x,\cdot)$ on $T^{1}_{x}M$ has measure zero.
\end{proof}

\begin{lemma}\label{lem_5_4}
    Suppose $\bm{v}\in T^{1}_{x}M$ is singular, then $q_J(x,\bm{v})=0$.
\end{lemma}

\begin{proof}
    By definition, $\bm{v}$ is singular, we can find a nonzero parallel normal Jacobi field $Y$ with $\bm{w}=Y(0)$. Then, $J^{\perp}_{\bm{v}}(\bm{w})=0$, $\ker J^{\perp}_{\bm{v}}$ is non-trivial, $q_J(x,\bm{v})=0$.
\end{proof}

We are ready to prove Theorem~\ref{thm_2_4}:

\manifold*

\begin{proof}
    By Lemma~\ref{lem_5_4}, $\sing$ is a subset of the zero set $Z(q_{J})$ of $q_J$. Therefore, we have
    \begin{displaymath}
        \begin{aligned}
            \mu_{L}(\sing) & \leq \mu_{L}(Z(q_{J}))                                                                                                                   \\
                           & =\dfrac{1}{\vol(M)}\int_{M}\sigma_{x}(Z(q_{J})\cap T^{1}_{x}M)\,\mathrm{d}\vol(x)                                                        \\
                           & =\dfrac{1}{\vol(M)}\left(\int_{\mathcal{D}}+\int_{M\setminus \mathcal{D}}\right) \sigma_{x}(Z(q_{J})\cap T^{1}_{x}M)\,\mathrm{d}\vol(x).
        \end{aligned}
    \end{displaymath}

    By Lemma~\ref{lem_5_3}, for any $x\in M\setminus \mathcal{D}$, $\sigma_{x}(Z(q_{J})\cap T^{1}_{x}M)=0$, thus the integration is also $0$. For $x\in \mathcal{D}$, we have $q_{J}(x,\cdot)=0$ on the whole $T^{1}_{x}M$, with the fiber measure of the zero set $\sigma_{x}(Z(q_{J})\cap T^{1}_{x}M)=1$. Combining these observations, we obtain that
    \begin{displaymath}
        \mu_{L}(\sing)\le \dfrac{1}{\vol(M)}\int_{\mathcal{D}}1\,\mathrm{d}\vol(x)=\dfrac{\vol(\mathcal{D})}{\vol(M)}.
    \end{displaymath}

    By Lemma~\ref{lem_5_2}, $\vol(\mathcal{D}\setminus F^{J}_{\infty})=0$. Therefore, we have
    \begin{displaymath}
        \mu_{L}(\sing)\leq\dfrac{\vol(F^{J}_{\infty})}{\vol(M)}.
    \end{displaymath}

    In particular, if $F^{J}_{\infty}$ has zero measure, then the geodesic flow is ergodic.
\end{proof}

\section{Ergodicity for Virtually Abelian Exceptional Regions}\label{sec6}
Throughout this section, let $(M,g)$ be a smooth closed rank-one manifold of nonpositive curvature, and $p:\widetilde{M}\to M$ be the universal cover with the group of deck transformations $\Gamma\simeq\pi_{1}(M)$. By the Cartan-Hadamard theorem, $\widetilde{M}$ is a complete, simply connected manifold of nonpositive sectional curvature. The action of $\Gamma$ on $\widetilde{M}$ is free, properly discontinuous, and cocompact. Further, $\Gamma$ is torsion-free.

We will retain the notation $\phi_t$, $c_{\bm{v}}$, $\sigma_{x}$ and $\mu_{L}$ as in Section~\ref{sec5}. The criterion below concerns exceptional regions whose fundamental groups have virtually Abelian images in $\pi_{1}(M)$, which gives a geometric way to control the singular set even when the flat locus has positive volume.

\begin{definition}[Trap]\label{def_6_1}
    A compact smooth region in $M$ is a connected, compact, embedded $n$-dimensional submanifold $\Omega\subset M$ with smooth boundary. We define its trapped set by

    \begin{displaymath}
        \Trap(\Omega) =\{\bm{v}\in T^{1}M\mid c_{\bm{v}}(t)\in \Omega,\quad\forall t\in\mathbb{R}\}.
    \end{displaymath}
\end{definition}

Note that $\Omega$ is compact, the continuity of the flow gives that
\begin{displaymath}
    \Trap(\Omega)=\bigcap_{t\in\mathbb{Q}}\phi_{-t}(\pi^{-1}(\Omega)).
\end{displaymath}

Therefore, the trapped set of $\Omega$ is a closed and $\phi_{t}$-invariant set.

\begin{definition}[Virtually Abelian]\label{def_6_2}
    A group $H$ is called virtually Abelian, if there is an Abelian subgroup $A\le H$ with finite index $[H:A]=m<\infty$.
\end{definition}

Here the subgroup $A$ need not be normal, but we can always find a normal Abelian subgroup of finite index.

\begin{lemma}\label{lem_6_1}
    Suppose $H$ is virtually Abelian, then $H$ contains a normal Abelian subgroup $H_{0}$ of finite index.
\end{lemma}

\begin{proof}
    Choose an Abelian subgroup $A\le H$ of finite index $m$ and consider the action of $H$ by the left multiplication on the left coset space $H/A$. This action naturally induces a homomorphism from $H$ to the permutation group $S_{m}$, whose kernel is given by
    \begin{displaymath}
        H_0=\bigcap_{h\in H}hAh^{-1}.
    \end{displaymath}

    Therefore, $H_{0}$ is a normal subgroup of $H$. It is Abelian because it is a subgroup of $A$. The index of $H_{0}$ is at most $m!$, which is finite. We call $H_{0}$ the core of $A$ in $H$.
\end{proof}

The main result we will prove in this section is Theorem~\ref{thm_2_5}:

\virtuallyabelian*

As mentioned before, we call these compact regions exceptional regions because their union contains all points at which strict negativity of sectional curvature fails. The subgroup $H_{i}$ is understood after choosing a base point and a path to a fixed base point of $M$. Different choices conjugate $H_{i}$, so the assumption does not depend on the choices. In particular, we do not require the map $\pi_{1}(\Omega_{i})\to \pi_{1}(M)$ to be injective. We will first explain the geometry behind the virtually Abelian condition.

\begin{lemma}\label{lem_6_2}
    If $H=\im\left(\pi_{1}(\Omega)\to\pi_{1}(M)\right)$ is virtually Abelian, then $H$ is finitely generated with a normal subgroup $H_{0}\simeq \mathbb{Z}^{k}$ ($k\geq0$) of finite index.
\end{lemma}

\begin{proof}
    A compact smooth region has the homotopy type of a finite CW complex, hence $\pi_{1}(\Omega)$ and its image $H$ are finitely generated. Choose a normal Abelian subgroup $H_0$ of finite index in $H$ as in Lemma~\ref{lem_6_1}. By Schreier's theorem, $H_0$ is finitely generated. Since $\pi_1(M)$ is torsion-free, so is $H_0$. The structure theorem for
    finitely generated Abelian groups gives $H_0\simeq\mathbb Z^k$ for some $k\geq0$.
\end{proof}

We call $k$ the virtual Abelian rank of $H$.

\begin{lemma}\label{lem_6_3}
    Let $\Omega\subset M$ be a compact smooth region, and $\widetilde{\Omega}$ be a connected component of the lift $p^{-1}(\Omega)$. After compatible choices of base points, write
    \begin{displaymath}
        H=\im\left(\pi_1(\Omega)\longrightarrow\pi_1(M)\right)\subset\Gamma.
    \end{displaymath}
    Then $H$ is the stabilizer of $\widetilde{\Omega}$:
    \begin{displaymath}
        H=\{\gamma\in\Gamma\mid\gamma\widetilde{\Omega}=\widetilde{\Omega}\},\qquad \widetilde{\Omega}/H\simeq\Omega.
    \end{displaymath}
    If $H_0$ is a normal subgroup of finite index in $H$, then $\widetilde{\Omega}/H_0\to\Omega$ is a finite Riemannian covering. In particular, $\widetilde{\Omega}/H_0$ is compact.
\end{lemma}

\begin{proof}
    Choose $x_0\in\Omega$ and its lift $\widetilde{x}_0\in\widetilde{\Omega}$. We identify the class of a loop based at $x_0$ with the deck transformation carrying $\widetilde{x}_0$ to the endpoint of its lift. A loop contained in $\Omega$ lifts to a path in $p^{-1}(\Omega)$ starting at $\widetilde{x}_0$. Its endpoint is therefore still in $\widetilde{\Omega}$, so the corresponding deck transformation preserves this component, $H\subset\operatorname{Stab_{\Gamma}}(\widetilde{\Omega})$.

    Conversely, if $\gamma\widetilde{\Omega}=\widetilde{\Omega}$, choose a path in $\widetilde{\Omega}$ from $\widetilde{x}_0$ to $\gamma\widetilde{x}_0$. Its projection is a loop in $\Omega$ and represents $\gamma$. Such a path exists because components of a manifold with boundary are path connected. This proves the equality of $H$ with the stabilizer.

    The restriction $p|_{\widetilde{\Omega}}$ is a covering of $\Omega$: its local covering charts are obtained by restricting evenly covered neighborhoods in $M$, and its surjectivity follows by lifting paths in $\Omega$. If two points of $\widetilde{\Omega}$ project to the same point, a unique deck transformation carries one to the other. That transformation preserves the component and hence belongs to $H$. Thus the fibers are exactly the $H$-orbits, proving $\widetilde{\Omega}/H\simeq\Omega$.

    Finally, passing from $H$ to $H_0$ gives a covering of degree $[H:H_0]$. Since the actions are isometric, it is a Riemannian covering. A finite covering of the compact manifold with boundary $\Omega$ is compact.
\end{proof}

\begin{lemma}[Semisimplicity of the Deck Transformations]\label{lem_6_4}
    Every $\gamma\in\Gamma$ is semisimple: the function $x\mapsto d(x,\gamma x)$ attains its infimum on $\widetilde{M}$.
\end{lemma}

\begin{proof}
    Since $M$ is closed and has nonpositive sectional curvature, it has no conjugate points. By \cite[Lemma~2.1]{CrokeSchroeder1986} and \cite[Proposition II.6.10]{BridsonHaefliger1999}, the displacement function of every nontrivial $\gamma\in\Gamma$ attains its minimum on $\widetilde{M}$.
\end{proof}

\begin{proposition}\label{prop_6_1}
    Let $\Omega$, $\widetilde{\Omega}$, and $H$ be as in Lemma~\ref{lem_6_3}. Suppose $H$ is virtually Abelian of rank $k$. Then there is a complete totally geodesic flat submanifold $F\subset \widetilde{M}$, isometric to $\mathbb{R}^k$, and a constant $C_{\Omega}$ such that
    \begin{displaymath}
        d(x,F)\leq C_{\Omega}\qquad \text{for every }x\in\widetilde{\Omega}.
    \end{displaymath}
    When $k=0$, $F$ is understood to be a single point.
\end{proposition}

\begin{proof}
    Choose the subgroup $H_{0}\simeq\mathbb{Z}^{k}$ as in Lemma~\ref{lem_6_2}. When $k\geq 1$, the action of $H_{0}$ on $\widetilde{M}$ is proper and semisimple as in Lemma~\ref{lem_6_4}. The flat torus theorem gives an $H_{0}$-invariant flat submanifold $F\simeq\mathbb{R}^{k}$, on which $H_{0}$ acts by translations with compact quotient (cf.~\cite[Theorem~II.7.1]{BridsonHaefliger1999}). If $k=0$, take any point for $F$, which is invariant under the trivial group $H_{0}$.

    Since $F$ is $H_{0}$-invariant, $x\mapsto d(x,F)$ descends to a continuous function on $\widetilde{\Omega}/H_{0}$. This quotient is compact by Lemma~\ref{lem_6_3}, so this function is bounded.
\end{proof}

This property actually reveals the geometry behind the virtually Abelian condition. The next is to estimate the distance to a totally geodesic submanifold. This will control the directions of trapped geodesics. Let $Y\subset \widetilde{M}$
be a complete, connected, embedded, totally geodesic submanifold of dimension $k$, thus, it is closed and convex. For $x\in \widetilde{M}$, let $q(x)$ be its nearest point in $Y$, and let $P_{q(x)\to x}$ denote parallel transport along the unique geodesic from $q(x)$ to $x$. Define
\begin{equation}
    r(x)=d(x,Y),\quad f(x)=\dfrac{1}{2} {d(x,Y)}^2,\quad\mathcal{E}_x=P_{q(x)\to x}(T_{q(x)}Y).
\end{equation}

For any $y\in Y$, denote its normal space to $Y$ by
\begin{displaymath}
    N_{y}Y={(T_{y}Y)}^{\perp}\coloneq\{\bm{v}\in T_{y}\widetilde{M}\mid \langle \bm{v},\bm{w}\rangle=0,\quad\forall\bm{w}\in T_{y}Y\},
\end{displaymath}
and denote the normal bundle of $Y$ by
\begin{displaymath}
    NY\coloneq\bigcup_{y\in Y}N_{y}Y.
\end{displaymath}

\begin{lemma}\label{lem_6_5}
    The function $r(x)=d(x,Y)$ is convex along every geodesic in $\widetilde{M}$.
\end{lemma}

\begin{proof}
    Let $c:[0,1]\to\widetilde M$ be an affinely parametrized geodesic, and $a:[0,1]\to Y$ be the affinely parametrized geodesic joining $q(c(0))$ to $q(c(1))$.

    The convexity of the distance between two geodesics in a Hadamard manifold gives that for any $t\in [0,1]$,
    \begin{displaymath}
        r(c(t))=d(c(t),Y)\leq d(c(t),a(t))\leq (1-t)r(c(0))+t\cdot r(c(1)).
    \end{displaymath}

    This is the required convexity.
\end{proof}

\begin{lemma}\label{lem_6_6}
    The nearest point projection $q:\widetilde{M}\to Y$ is smooth, and the normal exponential map
    \begin{displaymath}
        \exp^{\perp}:NY\longrightarrow\widetilde{M},\qquad (y,\bm{v})\to\exp_{y}(\bm{v})
    \end{displaymath}
    is a diffeomorphism.

    In these coordinates, $q(\exp^{\perp}(y,\bm{v}))=y$, and $f(\exp^{\perp}(y,\bm{v}))=\dfrac{1}{2}{\lVert \bm{v}\rVert}^2$. Thus, $f$ is smooth on $\widetilde{M}$ and $\mathcal{E}$ is a smooth subbundle of $T\widetilde{M}$ of rank $k$.
\end{lemma}

\begin{proof}
    Since $\widetilde{M}$ is a Hadamard manifold and $Y$ is complete, connected, and totally geodesic, $Y$ is closed and totally convex. By \cite[Lemma~3.1]{BishopONeill1969}, the normal exponential map $\exp^\perp:NY\to \widetilde{M}$ is a diffeomorphism.

    Let $\pi_{NY}:NY\to Y$ denote the bundle projection. First variation shows that a minimizing geodesic from $Y$ to $x\in\widetilde{M}$ is normal to $Y$. Uniqueness under $\exp^{\perp}$ gives
    \begin{displaymath}
        q\left(\exp^{\perp}(y,\bm{v})\right)=y, \qquad d\left(\exp^{\perp}(y,\bm{v}),Y\right)=\|\bm{v}\|.
    \end{displaymath}
    Hence,
    \begin{displaymath}
        q=\pi_{NY}\circ{(\exp^{\perp})}^{-1},\qquad f\left(\exp^\perp(y,\bm{v})\right)=\frac12{\|\bm{v}\|}^2.
    \end{displaymath}
    These formulas show that $q$ and $f$ are smooth. Finally, smooth dependence of parallel transport along the normal geodesics implies that $\mathcal{E}$ is a smooth rank $k$ subbundle of $T\widetilde{M}$.
\end{proof}

\begin{lemma}\label{lem_6_7}
    For each $x\in\widetilde{M}$ and $\bm{w}\in T_{x}\widetilde{M}$, write $\bm{w}=\bm{w}_{\parallel}+\bm{w}_{\perp}$ according to the orthogonal decomposition
    $T_x{\widetilde{M}}=\mathcal{E}_{x}\oplus\mathcal{E}_{x}^{\perp}$. Then
    \begin{equation}
        \Hess_{x} f(\bm{w},\bm{w}) \geq {\|\bm{w}_{\perp}\|}^2.
    \end{equation}
    In particular, $\Hess f$ is positive semidefinite and
    \begin{displaymath}
        \ker\Hess_{x} f\subset\mathcal{E}_x.
    \end{displaymath}
\end{lemma}

\begin{proof}
    First assume $x\notin Y$. Let $x(s)$ be a geodesic with $x(0)=x$ and $x'(0)=\bm{w}$, and put $y(s)=q(x(s))$.

    Denote $(y(s),\bm{v}(s))={(\exp^{\perp})}^{-1}(x(s))$. By Lemma~\ref{lem_6_6}, the normal geodesics
    \begin{displaymath}
        \sigma_s(t)=\exp_{y(s)}(t\bm{v}(s)), \qquad 0\leq t\leq 1,
    \end{displaymath}
    from $y(s)$ to $x(s)$ form a smooth geodesic variation.

    Set $\sigma=\sigma_0$, $T=\partial_t\sigma_s|_{s=0}$, and $J=\partial_s\sigma_s|_{s=0}$. Since the parameter interval has length one, the energy of $\sigma_s$ equals $f(x(s))$.

    The second variation formula gives
    \begin{displaymath}
        \begin{aligned}
            \Hess_{x} f(\bm{w},\bm{w})&=\int_0^1\left({\|D_tJ\|}^2-\langle R(J,T)T,J\rangle\right)\,\mathrm{d}t+{[\langle D_s\partial_s\sigma_s,T\rangle]}_{t=0}^{1}\\
                                      &=\int_0^1\left({\|D_tJ\|}^2-\langle R(J,T)T,J\rangle\right)\,\mathrm{d}t.
        \end{aligned}
    \end{displaymath}

    Indeed, the acceleration at $t=1$ is zero because $x(s)$ is a geodesic. At $t=0$, the acceleration of $y(s)$ is tangent to $Y$ because $Y$ is totally geodesic, whereas $T(0)$ is normal to $Y$. Thus both endpoint terms vanish. Then, nonpositive curvature implies
    \begin{displaymath}
        \Hess_{x} f(\bm{w},\bm{w})\geq\int_0^1{\|D_tJ\|}^2\,\mathrm{d}t.
    \end{displaymath}

    Let $P_{t\to 1}$ denote parallel transport along $\sigma$ from $\sigma(t)$ to $x$. All vectors in the following identity lie in $T_{x}\widetilde{M}$:
    \begin{displaymath}
        \bm{w}-P_{0\to 1}J(0)=\int_0^1 P_{t\to 1}(D_{t}J(t))\,\mathrm{d}t.
    \end{displaymath}
    Cauchy--Schwarz inequality therefore yields
    \begin{displaymath}
        \int_0^1 {\|D_tJ\|}^2\,\mathrm{d}t\geq{\|\bm{w}-P_{0\to 1}J(0)\|}^2\geq{\|\bm{w}_{\perp}\|}^2,
    \end{displaymath}
    since $J(0)=y'(0)\in T_{q(x)}Y$ and $P_{0\to 1}J(0)\in\mathcal{E}_x$.

    If $x\in Y$, use the normal bundle coordinates of Lemma~\ref{lem_6_6}. At $(x,0)$ we have $df_x=0$ and $f(y,\bm{v})=\frac12\|\bm{v}\|^2$. The differential of $\exp^\perp$ at $(x,0)$ identifies the tangent and normal factors with $T_{x}Y$ and ${(T_{x}Y)}^\perp$. Consequently,
    \begin{displaymath}
        \Hess_{x} f(\bm{w},\bm{w})={\|\bm{w}_{\perp}\|}^2.
    \end{displaymath}

    This proves the estimate everywhere. For a positive semidefinite symmetric form, its kernel consists of the vectors on which its quadratic form vanishes. Thus, $\ker\Hess_{x}f\subset\mathcal{E}_{x}$.
\end{proof}

Now we are ready to prove our criterion.

\begin{theorem}\label{thm_6_1}
    Let $\Omega\subset M$ be a compact smooth region. Suppose
    \begin{displaymath}
        H=\im\left(\pi_1(\Omega)\longrightarrow\pi_1(M)\right)
    \end{displaymath}
    is virtually Abelian of rank $k<n$. Then $\mu_L(\Trap(\Omega))=0$.

    If $k\geq 1$, then
    \begin{displaymath}
        \dim_H\Trap(\Omega)\leq n+k-1.
    \end{displaymath}

    If $k=0$, then $\Trap(\Omega)=\emptyset$.
\end{theorem}

\begin{proof}
    Choose a component $\widetilde{\Omega}$ of $p^{-1}(\Omega)$ and the flat $F\simeq\mathbb{R}^k$ given by Proposition~\ref{prop_6_1}. Write $f(x)=\frac{1}{2} {d(x,F)}^2$ and $\mathcal{E}_{x}=P_{q(x)\to x}(T_{q(x)}F)$.

    Let $\widetilde{c}:\mathbb{R}\to \widetilde{M}$ be a unit-speed geodesic contained in $\widetilde{\Omega}$. By Lemma~\ref{lem_6_5}, the function $t\mapsto d(\widetilde{c}(t),F)$ is convex. It is bounded above by Proposition~\ref{prop_6_1}, and a convex function on the whole real line that is bounded above is therefore constant.

    Thus, $f\circ\widetilde{c}$ is constant, and
    \begin{displaymath}
        0=(f\circ\widetilde{c})''(t)=\Hess_{\widetilde{c}(t)}f(\widetilde{c}'(t),\widetilde{c}'(t)).
    \end{displaymath}
    Lemma~\ref{lem_6_7} gives
    \begin{equation}\label{eq_6_1}
        \widetilde{c}'(t)\in\mathcal{E}_{\widetilde{c}(t)}\qquad\forall t\in\mathbb{R}.
    \end{equation}

    If $k=0$, the bundle $\mathcal{E}$ has zero-dimensional fibers, which contain no unit vectors. Therefore there is no such geodesic.

    When $1\leq k<n$, the unit sphere bundle
    \begin{displaymath}
        T^{1}(\mathcal{E})=\{\bm{w}\in T^{1}\widetilde{M}\mid\bm{w}\in\mathcal{E}_{\pi_{\widetilde{M}}(\bm{w})}\}
    \end{displaymath}
    is a smooth submanifold of $T^{1}\widetilde{M}$ of dimension $n+k-1$, where $\pi_{\widetilde{M}}:T^1\widetilde{M}\to \widetilde{M}$ is the projection. Equation~\eqref{eq_6_1} shows that $\Trap_{\widetilde{M}}(\widetilde{\Omega})\subset T^{1}(\mathcal{E})$.

    At each $x\in \widetilde{M}$, the fiber of $T^{1}(\mathcal{E})$ is a sphere of dimension $k-1$ in $T_{x}^1\widetilde{M}\simeq S^{n-1}$. Since $k<n$, it has zero angular measure. The Liouville measure on $T^1\widetilde{M}$ disintegrates as Riemannian volume on $\widetilde{M}$ times angular measure on the fibers, up to a fixed positive normalization. Thus $T^{1}(\mathcal{E})$ has zero Liouville measure in $T^{1}\widetilde{M}$.

    The induced map $Dp:T^1\widetilde{M}\to T^1M$ is a covering and a local isometry. Every geodesic trapped in $\Omega$ has a lift trapped in $\widetilde{\Omega}$. Conversely, projecting a geodesic trapped in $\widetilde{\Omega}$ gives a geodesic trapped in $\Omega$. Hence
    \begin{displaymath}
        Dp\left(\Trap_{\widetilde{M}}(\widetilde{\Omega})\right)=\Trap(\Omega).
    \end{displaymath}

    Cover $T^{1}\widetilde{M}$ by countably many coordinate neighborhoods on which $Dp$ is a diffeomorphism. On each such neighborhood it preserves Liouville null sets and does not increase Hausdorff dimension, by restriction to countably many relatively compact charts. Countable unions therefore give
    \begin{displaymath}
        \mu_L(\Trap(\Omega))=0,\qquad\dim_{H}\Trap(\Omega)\leq n+k-1.
    \end{displaymath}
\end{proof}

\begin{remark}
    The trapped-set estimate requires no additional curvature assumptions on the exceptional region $\Omega$ or its complement. Thus, we expect this estimate to have further applications to study the singular set. 
\end{remark}

\begin{proof}[Proof of Theorem~\ref{thm_2_5}]
    If $\bm{v}\in\sing$, then the geodesic along $\bm{v}$ cannot pass through a point at which every sectional curvature is strictly negative. It follows that $c_{\bm{v}}(\mathbb{R})\subset\bigcup_{i=1}^N\Omega_i$. Since the finitely many compact sets $\Omega_{i}$ are pairwise disjoint and $c_{\bm v}(\mathbb R)$ is connected, the entire geodesic lies in a single $\Omega_{i}$. Therefore,
    \begin{displaymath}
        \sing\subset\bigcup_{i=1}^{N}\Trap(\Omega_{i}).
    \end{displaymath}

    By Theorem~\ref{thm_6_1}, each $\Trap(\Omega_i)$ has zero Liouville measure, and hence, the singular set has zero Liouville measure.

    We also have that
    \begin{displaymath}
        \dim_{H}(\sing)\leq\max_{i}(n+k_{i}-1)\leq 2n-2.
    \end{displaymath}

    If all $k_{i}$ vanish, the singular set is empty, which yields that the geodesic flow is Anosov.
\end{proof}

As an application, consider Heintze’s example with a flat product collar $\Omega=T^{n-1}\times[-a,a]$, and strictly negative sectional curvature on $M\setminus\Omega$~\cite[Example 4]{BallmannBrinEberlein1985}. The inclusion $\Omega\hookrightarrow M$ induces an injective homomorphism on fundamental groups. Indeed, every nontrivial element of $\pi_1(\Omega)$ is represented by a nonconstant straight closed geodesic in the middle torus. If this geodesic were null-homotopic in $M$, it would lift to a nonconstant closed geodesic in the Hadamard universal cover, which is impossible. Hence
\begin{displaymath}
    H=\im\left(\pi_{1}(\Omega)\to\pi_{1}(M)\right)\cong\mathbb Z^{n-1}, \qquad k=n-1<n.
\end{displaymath}

The geometric criterion therefore gives $\mu_L(\sing)=0$ and $\dim_{H}(\sing)\leq 2n-2$, and the geodesic flow is ergodic with respect to Liouville measure. Moreover, every singular geodesic is trapped in $\Omega$, where its interval coordinate is affine and bounded, hence constant. Conversely, every geodesic tangent to a torus slice remains in the flat collar and is singular. Thus

\begin{displaymath}
    \sing\simeq T^{n-1}\times[-a,a]\times S^{n-2},\qquad\dim_H(\sing)=2n-2.
\end{displaymath}

This example attains the dimension bound, and its geodesic flow is not Anosov because $\sing\neq\emptyset$.

\section{Ergodicity on Real Analytic Rank One Manifolds}\label{sec7}
We now prove Theorem~\ref{thm_2_6} in the real-analytic setting.

\analymlfd*

Again, our target is to show the singular set has zero Liouville measure. We use the notation of Section~\ref{sec5}. Let $x\in M$ and $\bm{v}\in T^{1}_{x}M$. Recall that $E_{\bm{v}}$ denotes the orthogonal complement of $\bm{v}$ and $P_{\bm{v},t}:E_{\bm{v}}\to E_{\phi_{t}(\bm{v})}$ denotes the parallel transport along $c_{\bm{v}}$.

\begin{definition}\label{def_7_1}
    On $E_{\bm{v}}$ we define the parallel-transport pullback of the curvature endomorphism $A(\bm{v},t):E_{\bm{v}}\to E_{\bm{v}}$ by
    \begin{equation}
        A(\bm{v},t)(\bm{w})=P^{-1}_{\bm{v},t}(R(P_{\bm{v},t}(\bm{w}),c'_{\bm{v}}(t))c'_{\bm{v}}(t)).
    \end{equation}

    Then, $A(\bm{v},t)$ is a self-adjoint linear operator.
\end{definition}

\begin{proof}
    We first show that this operator is well-defined. Fix $\bm{v}$ and choose $\bm{w}\in E_{\bm{v}}$, we need to show $A(\bm{v},t)(\bm{w})\in E_{\bm{v}}$ as well. Denote $\bm{w}_t=P_{\bm{v},t}\bm{w}$, the curvature symmetries imply
    \begin{displaymath}
        \langle R(\bm{w}_t,c'_{\bm{v}}(t))c'_{\bm{v}}(t),c'_{\bm{v}}(t)\rangle=0.
    \end{displaymath}
    Thus, the curvature vector lies in $E_{\phi_{t}(\bm{v})}$, so $A(\bm{v},t)$ is well-defined.

    Linearity follows immediately from the fact that $P_{\bm{v},t}$, its inverse, and the curvature endomorphism are all linear maps. To verify that it is self-adjoint, choose $\bm{u},\bm{w}\in E_{\bm{v}}$ and denote $\bm{u}_t=P_{\bm{v},t}(\bm{u})$ and $\bm{w}_t=P_{\bm{v},t}(\bm{w})$.

    \begin{displaymath}
        \begin{aligned}
            \langle A(\bm{v},t)(\bm{u}),\bm{w}\rangle & =\langle P^{-1}_{\bm{v},t}(R(\bm{u}_t,c'_{\bm{v}}(t))c'_{\bm{v}}(t)),\bm{w}\rangle \\
                                                      & =\langle R(\bm{u}_t,c'_{\bm{v}}(t))c'_{\bm{v}}(t),\bm{w}_{t}\rangle                    \\
                                                      & =R(\bm{u}_t, c'_{\bm{v}}(t),c'_{\bm{v}}(t),\bm{w}_t)                                   \\
                                                      & =R(\bm{w}_t,c'_{\bm{v}}(t),c'_{\bm{v}}(t),\bm{u}_t)                                    \\
                                                      & =\langle A(\bm{v},t)\bm{w},\bm{u}\rangle.
        \end{aligned}
    \end{displaymath}

    Therefore, $A^{\ast}(\bm{v},t)=A(\bm{v},t)$ is self-adjoint.
\end{proof}

\begin{definition}\label{def_7_2}
    For any $\bm{v}\in T^{1}M$, define the operator $Q(\bm{v}):E_{\bm{v}}\to E_{\bm{v}}$ by
    \begin{displaymath}
        Q(\bm{v})(\bm{w})=\int_{-1}^{1} A^{\ast}(\bm{v},t)\circ A(\bm{v},t)(\bm{w})\,\mathrm{d}t.
    \end{displaymath}

    Then, $Q(\bm{v})$ is self-adjoint and positive semidefinite.
\end{definition}

\begin{proof}
    Choose any $\bm{w}\in E_{\bm{v}}$. We can see that $Q(\bm{v})$ is well-defined from $E_{\bm{v}}$ to itself. This is because $A(\bm{v},t)$ preserves $E_{\bm{v}}$, and so does $A^{\ast}(\bm{v},t)\circ A(\bm{v},t)$. This holds for any $t\in [-1,1]$ and hence, the integral also lies in $E_{\bm{v}}$. $A^{\ast}(\bm{v},t)A(\bm{v},t)$ is self-adjoint for all $t$, and so is $Q(\bm{v})$. To verify positive semidefiniteness, we compute 

    \begin{displaymath}
        \begin{aligned}
            \langle Q(\bm{v})(\bm{w}),\bm{w}\rangle & =\langle \int_{-1}^{1} A^{\ast}(\bm{v},t)\circ A(\bm{v},t)(\bm{w})\,\mathrm{d}t,\bm{w}\rangle \\
                                                    & =\int_{-1}^{1}\langle A^{\ast}(\bm{v},t)\circ A(\bm{v},t)(\bm{w}),\bm{w}\rangle\,\mathrm{d}t  \\
                                                    & =\int_{-1}^{1}\langle A(\bm{v},t)(\bm{w}),A(\bm{v},t)(\bm{w})\rangle\,\mathrm{d}t             \\
                                                    & =\int_{-1}^{1}{\lVert A(\bm{v},t)(\bm{w})\rVert}^2\,\mathrm{d}t\geq 0.
        \end{aligned}
    \end{displaymath}
\end{proof}

\begin{definition}\label{def_7_3}
    Define the function $D:T^{1}M\to\mathbb{R}$ by
    \begin{equation}
        D(\bm{v})\coloneq \det Q(\bm{v}).
    \end{equation}

    From Definition~\ref{def_7_2}, we have $Q(\bm{v})$ positive semidefinite, thus, $D(\bm{v})\ge 0$.
\end{definition}

One important property is that in the real-analytic setting, the zero set of $D$ defined in Definition~\ref{def_7_3} characterizes the singular set of $T^{1}M$.

\begin{theorem}\label{thm_7_1}
    A vector $\bm v\in T^1M$ is singular if and only if $D(\bm{v})=0$.
\end{theorem}

\begin{proof}
    Suppose that $\bm{v}$ lies in the singular set, $\rank(\bm{v})\ge 2$. By the definition of rank, we can find a nonzero parallel normal Jacobi field $Y$ along $c_{\bm{v}}$. Denote $\bm{w}=Y(0)\in E_{\bm{v}}$ which is a nonzero vector and $Y(t)=P_{\bm{v},t}(\bm{w})$.

    $Y(t)$ satisfies the Jacobi equation, that is
    \begin{displaymath}
        \dfrac{\mathrm{D}^2 Y}{\mathrm{d}t^{2}}+R(Y(t),c'_{\bm{v}}(t))c'_{\bm{v}}(t)=0.
    \end{displaymath}

    Since $Y(t)$ is parallel, $\dfrac{\mathrm{D}Y}{\mathrm{d}t}(t)=0$, thus, $\dfrac{\mathrm{D}^{2}Y}{\mathrm{d}t^2}(t)=0$. Therefore, we get
    \begin{displaymath}
        R(Y(t),c'_{\bm{v}}(t))c'_{\bm{v}}(t)=0.
    \end{displaymath}

    Hence,
    \begin{displaymath}
        A(\bm{v},t)(\bm{w})=P^{-1}_{\bm{v},t}(R(P_{\bm{v},t}(\bm{w}),c'_{\bm{v}}(t))c'_{\bm{v}}(t))=0,
    \end{displaymath}
    $Q(\bm{v})(\bm{w})=0$ and $D(\bm{v})=0$.

    Conversely, suppose that $D(\bm{v})=0$, then $\ker Q(\bm{v})$ is non-trivial. Choose a nonzero vector $\bm{w}\in \ker Q(\bm{v})\subset E_{\bm{v}}$, we have
    \begin{displaymath}
        0=\langle Q(\bm{v})(\bm{w}),\bm{w}\rangle=\int_{-1}^{1} {\lVert A(\bm{v},t)(\bm{w})\rVert}^{2}\,\mathrm{d}t.
    \end{displaymath}

    $A(\bm{v},t)(\bm{w})$ is continuous, this implies $A(\bm{v},t)(\bm{w})=0$ for $t\in [-1,1]$. Since the metric is real analytic, geodesics and parallel transport depend real analytically on time. Thus, $t\mapsto A(\bm{v},t)\bm{w}$ is real analytic on $\mathbb{R}$. By the unique continuation of the analytic function, we get
    \begin{displaymath}
        A(\bm{v},t)(\bm{w})=0,\qquad t\in\mathbb{R}.
    \end{displaymath}

    Now let $Y(t)=P_{\bm{v},t}(\bm{w})$, from the above argument, we have
    \begin{displaymath}
        R(Y(t),c'_{\bm{v}}(t))c'_{\bm{v}}(t)=0,\qquad t\in\mathbb{R}.
    \end{displaymath}

    Since $Y$ is parallel, $D_{t}^{2}Y=0$. Together with the preceding curvature identity, this shows that $Y$ is a nonzero parallel normal Jacobi field along $c_{\bm{v}}$. Thus, $\bm{v}\in\sing$.
\end{proof}

Furthermore, we can evaluate the rank of the vector $\bm{v}$ by studying the operator $Q(\bm{v})$.

\begin{corollary}
    For a real-analytic metric,
    \begin{displaymath}
        \rank(\bm{v})=1+\dim\ker Q(\bm{v}).
    \end{displaymath}
\end{corollary}

We are ready to prove Theorem~\ref{thm_2_6} now.

\begin{proof}[Proof of Theorem~\ref{thm_2_6}]
    In local real-analytic frames of the normal bundle, the matrix entries of $A(\bm{v},t)$ are jointly real analytic in $(\bm{v},t)$. Integration over the fixed interval $[-1,1]$ therefore yields a real-analytic endomorphism $Q$. Consequently, $D=\det Q$ is a globally defined real analytic function on $T^{1}M$.

    Given $\rank(M)=1$, we can find some $\bm{v}\in T^{1}M$ with rank $1$. By Theorem~\ref{thm_7_1}, $D(\bm{v})>0$, so $D$ is not identically zero.

    Since $M$ is connected and $\dim M\geq2$, $T^1M$ is connected. The nontrivial real-analytic function $D$ therefore has a zero set of measure zero with respect to the Liouville measure. By Theorem~\ref{thm_7_1},   

    \begin{displaymath}
        \mu_{L}(\sing)=\mu_{L}(D^{-1}(0))=0,
    \end{displaymath}
    ergodicity of the geodesic flow follows.
\end{proof}

\begin{remark}
    If we only assume $C^{\infty}$-smoothness, what we can show is the singular set $\sing$ is a subset of the zero set $D^{-1}(0)$. The set of zeros of a smooth function may have positive measure, thus, this method cannot be applied to smooth manifolds directly.
\end{remark}

\section*{Acknowledgments}
The authors wish to express their gratitude to Professor Lan Wen for his guidance and support during their time working and studying at Peking University.

Fei Liu is partially supported by the National Natural Science Foundation of China under Grant No.~12571179.

Xiaokai Liu would like to thank Professor Raul Ures and Professor Jana Hertz for their invaluable guidance and continuous support at Southern University of Science and Technology.

\noindent{\bfseries Declaration of Interest:} The authors declare that they have no conflicts of interest.

\newpage
\printbibliography

\end{document}